\documentclass[reqno,11pt]{amsart}
\usepackage{amssymb}
\usepackage{amsmath}
\usepackage{amsthm}
\usepackage{pgf}
\usepackage{color}

\usepackage{graphicx}   
\usepackage{hyperref}
\usepackage{graphicx}
\usepackage{tikz}
\usepackage{pgf}

\newtheorem{theorem}{Theorem}[section]
\newtheorem{lemma}[theorem]{Lemma}
\newtheorem{proposition}[theorem]{Proposition}
\newtheorem{corollary}[theorem]{Corollary}

\newcommand{\Rz}{\mathbb{R}}
\newcommand{\Nz}{\mathbb{N}}
\newcommand{\Zz}{\mathbb{Z}}
 
\newcommand{\epsi}{\varepsilon}
\newcommand{\eps}{\varepsilon}
\newcommand{\haz}{\widehat}
 
\newcommand{\vt}{v_3}
\newcommand{\vl}{v_2}

\newcommand{\bx}{\boldsymbol  x}
\newcommand{\bv}{\boldsymbol  v}
\newcommand{\bp}{\boldsymbol  p}
\newcommand{\bq}{\boldsymbol  q}
\newcommand{\by}{\boldsymbol y}
\newcommand{\bvarphi}{\boldsymbol \varphi}
\newcommand{\calF}{\mathcal F}
\newcommand{\calS}{\mathcal S}
\newcommand{\cE}{\mathcal E}
\newcommand{\bT}{\boldsymbol T}
\newcommand{\bN}{\boldsymbol N} 

\theoremstyle{definition}
\begingroup
\newtheorem{definition}[theorem]{Definition}

\endgroup

\begin{document}

\title[ Stability of 
  $\Zz^2$ configurations  in 3d ]{ Stability of 
  $\Zz^2$ configurations  in 3d }

\author[L. B\'etermin]{Laurent B\'etermin} 
\address[Laurent B\'etermin]{Institut Camille Jordan, Universit\'e Claude Bernard Lyon 1, 21 avenue Claude Bernard, 69622 Villeurbanne Cedex, France 
}
\email{betermin@math.univ-lyon1.fr}
\urladdr{https://sites.google.com/site/homepagelaurentbetermin/}

\author[M. Friedrich]{Manuel Friedrich} 
\address[Manuel Friedrich]{Department of Mathematics, Friedrich-Alexander Universit\"at Erlangen-N\"urnberg. Cauerstr.~11,
D-91058 Erlangen, Germany, \& Mathematics M\"{u}nster,  
University of M\"{u}nster, Einsteinstr.~62, D-48149 M\"{u}nster, Germany}
\email{manuel.friedrich@fau.de}
\urladdr{https://www.uni-muenster.de/AMM/en/Friedrich/}

\author[U. Stefanelli]{Ulisse Stefanelli} 
\address[Ulisse Stefanelli]{Faculty of Mathematics, University of
  Vienna, Oskar-Morgenstern-Platz 1, A-1090 Vienna, Austria,
Vienna Research Platform on Accelerating
  Photoreaction Discovery, University of Vienna, W\"ahringerstra\ss e 17, 1090 Wien, Austria,
 \& Istituto di
  Matematica Applicata e Tecnologie Informatiche {\it E. Magenes}, via
  Ferrata 1, I-27100 Pavia, Italy
}
\email{ulisse.stefanelli@univie.ac.at}
\urladdr{http://www.mat.univie.ac.at/$\sim$stefanelli}


\begin{abstract}  Inspired by the issue of stability of
  molecular structures, we investigate the strict minimality of  
  point sets with respect to configurational energies featuring two-
  and  three-body  contributions.  Our main focus is on characterizing
  those configurations which cannot be  
deformed without changing distances between first neighbors or angles
formed by pairs of first neighbors. Such configurations are called
{\it angle-rigid}.

We tackle this question  in the class of  
finite configurations in $\Zz^2$, seen as planar three-dimensional
point sets. A
sufficient condition  preventing angle-rigidity  is presented. This
condition is also proved to be necessary when restricted to 
specific  subclasses  of configurations.  
\end{abstract}

                                %
\maketitle

\pagestyle{myheadings}

\section{Introduction}

 At moderate temperatures, atoms tend to form bonds and cluster together forming
molecules and crystals.  
The actual 3d arrangement of atoms in   such clusters has a
fundamental influence on its physical and mechanical
properties. Correspondingly, the 
prediction of such 3d geometries is of great importance
and a central object of study in chemistry and molecular biology.

The quantum nature of atomic bonding makes the
determination of 3d molecular geometries from first principles
 unfeasible  as the number of atoms scales  up.  A successful
model, capable of describing the geometry of  comparatively  large atomic
ensembles, is Molecular Mechanics  \cite{Friesecke-Theil15,Lewars}. There, configurations are
modeled as   collections of atomic positions and types, to which a {\it
  configurational energy} is associated. Such
energy is specified in terms of classical potentials and is 
phenomenological in nature. In particular, it usually features
two-body interactions, minimized at some preferred bond length, as
well as three-body terms,
favoring specific angles between bonds
\cite{Brenner,Stillinger-Weber85,Tersoff}.

In the case of identical atoms, under the classical Born-Oppenheimer
approximation,  the configuration
is purely determined by the atomic positions $\{\bx_1, \dots,\bx_n\}
\in \Rz^{3n}$, and
the configurational energy $\mathcal E=\mathcal E(\{\bx_1, \dots,\bx_n\})$. 
The latter is assumed to be 
decomposed as $\mathcal E=\mathcal E_2+\mathcal E_3$, where $\mathcal E_2$ and $\mathcal E_3$ correspond to two- and
three-body contributions, respectively. In particular, $\mathcal E_2$ is a
function of those pairs $(\bx_i,\bx_j)$ corresponding to bonded atoms, and
$\mathcal E_3$ depends on those triplets $(\bx_i,\bx_j,\bx_k)$ corresponding to
the angle formed by the bonds $(\bx_i,\bx_j)$ and
$(\bx_j,\bx_k)$.  Having in mind the specific case of covalent and ionic
bonding, we will neglect long-range
interactions in the following. In particular, $\mathcal E$ will feature
nearest- and (partially) next-to-nearest interaction only.  

A central issue with respect to molecular structuring is that of {\it
  stability}: actual 3d molecular arrangements need to show some
stability (thermodynamic, electrochemical, mechanical) in order to be
observed in experiments. In the context of Molecular Mechanics, this
calls for identifying strict  (local or global)  minimizers of the configurational energy.
In the setting of two- and three-body interactions described above, two
configurations with identical bond
lengths and bond angles necessarily have the 
 same energy. As such, in order to classify strict energy
minimizers, one is asked to identify those configurations which cannot be
infinitesimally perturbed (up to isometries) 
without changing bond  lengths  or bond angles. We call these
configurations {\it angle-rigid}, inspired by related rigidity
concepts in  graph theory,  see Section \ref{sec:discrete} below.
 Note that  the   concept of  angle-rigidity  has a clear  molecular-mechanical relevance,
for the occurrence of different molecular geometries  having same bond lengths
and angles is not uncommon and is usually referred to as {\it
  conformational isomerism}. Butane ($C_4H_{10}$) is a classical example
featuring more possible conformal isomers \cite{Allinger}.

The focus of this article is to investigate angle-rigidity for  a
specific class of  3d configurations   inspired by
crystallization \cite{Blanc},  see Figure~\ref{fig:initial}. 
\begin{figure}[h]
\centering
\pgfdeclareimage[width=125mm]{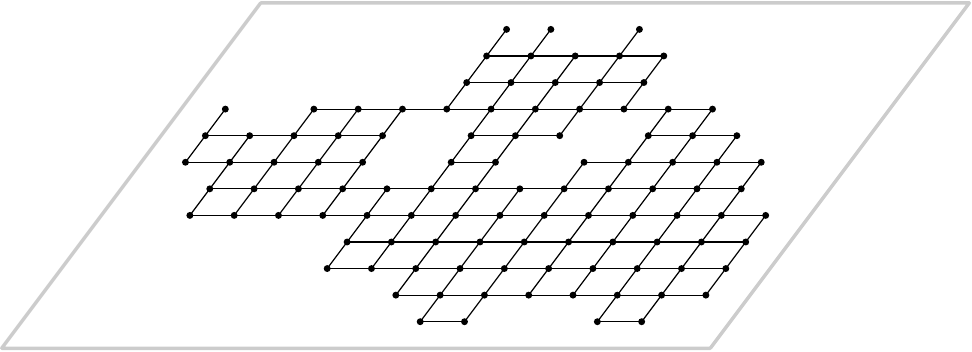}{initial}  
\pgfuseimage{initial} 
\caption{A $\Zz^2$ configuration}
\label{fig:initial}
\end{figure}  These are given as finite collections of
points in $\Zz^2\times \{0\}$, with bonds corresponding to first
neighbors. 
We advance  a  sufficient condition \eqref{suff} for 
strict minimality,  based on a  straightforward  combinatorial
argument. Such  sufficient  condition
 \eqref{suff}  is in general not necessary,  but turns out to fully
 characterize strict minimizers  when restricted to particular classes of 
configurations. We call such classes {\it foldings} and {\it
  shear-resistant  configurations},  inspired by the specific
geometry of their respective perturbations which do not change bonds
lengths nor bond angles.
 Note that these two classes do not cover all configurations: there
 are  non-strict minimizing  configurations which are neither
 foldings nor shear-resistant. On the other hand, every  strict
 minimizing 
 configurations is shear-resistant and thus  covered by the
 characterization (see also Figure \ref{fig:corollary}).  The
 relevance of our results in connection with applications in Molecular Mechanics is discussed
 in Section \ref{sec: appl} below.

 The paper is organized as follows. In Section \ref{sec:z2} we
 introduce notation and basic concepts. 
 Although our focus is on strict minimality of configurations, the
 setting of our problem is strongly reminiscent of rigidity in graph
 theory. We comment on this relation in  Section
 \ref{sec:discrete}. 
 The characterization of strict minimizers of two- and three-body-interaction
 configurational energies  in terms of angle-rigidity is given
 in Section \ref{sec: appl}. The sufficient condition
 for non-angle-rigidity is presented in Section \ref{sec: suff},
 together with the discussion of folding configurations.  Sections 
\ref{sec: cond}-\ref{sec:char}  then discuss the fine geometry of
shear-resistance configurations.

\section{Definition of angle-rigid configurations}\label{sec:z2}

 This section is devoted to introducing some  basic   notions  and
 notation,  leading to
the Definition~\ref{def: ar} of
 {\it angle-rigid $\Zz^2$ configurations}. In Section
\ref{sec:discrete} we collect some comments on the relation of this
notion with other rigidity  concepts  in graph theory.

\subsection{Basic notation and definition of  angle-rigid $\Zz^2$ configurations}\label{se:nota}

 Our analysis focuses on point configurations in three space
dimensions, namely  on  collections  $C = \{\bx_k \}_{k=1}^n$ of
$n$ distinct points in $\Rz^3$. Among these, we call {\it $\Zz^2$
  configurations} those with
$$ \bx_k \in \Zz^2=\{ (i,j,0) \ : \ i,\, j \in \Zz\}, \quad  \quad \text{ for all }   k\in \{1,...,n\},$$ 
see Figure \ref{fig:initial}, and indicate them as $C
\in\Zz^{2n}$ in the following. Note that  we  interpret
$\Zz^2$ as a proper subset of $\Rz^3$. In order to shorten notation,
whenever possible we refer to points in $\Zz^2$ by specifying their
first two integer coordinates only.  We call {\it axes} the subsets of $\Zz^2$ of the form
$\{i\} \times \Zz$ or $\Zz \times \{j\}$, for some  $i,\,j\in \Zz$.

To each configuration $C = \{\bx_k \}_{k=1}^n$ of
$n$ points in $\Rz^3$, we associate  the  segments joining 
pairs of points having distance smaller than $1+2\epsi$, for some
given  and fixed small parameter  $ 0 < \epsi \le  (\sqrt{2}-1)/4$. 
More precisely, we associate
to $C$ the set of  {\it neighbors} as specified by the pairs 
\begin{align}\label{eq: neighbor0}
N(C) =\{ (k,k')  \in  \{1,\dots,n\}^2\ \colon \  |\bx_{k} - \bx_{k'} | \le  1 + 2\eps  \}.
\end{align}
Note that each pair of neighbors is counted twice in
$N(C)$. If  $(k,{k'})  \in N(C)$,  
we call the straight segment   joining $\bx_k$ and $\bx_{k'}$ 
a   {\it bond}  and we say that the two points are {\it bonded}.
Bonds are  hence  identified with pairs $ (k,{k'}) \in
N(C)$, up to permutations.  For each
$\Zz^2$ configuration  one has that
$( k, {k'})\in N(C)$ iff  $\bx_k$ and $\bx_{k'}$ are at distance
$1$. In the following, we will use the notation $\bN(C)
=\{(\bx_k,\bx_{k'})\colon \, (k,k')\in N(C)\}$ to indicate the pairs of
bonded points.

We also indicate adjacent bonds by defining the  set of {\it triplets}  
\begin{align}\label{eq: neighbor000}
T(C) = \{( k, {k'}, {k''})  \in  \{1,\dots,n\}^3  \,
\colon \,  ( k, {k'}), \,
( {k'}, {k''}) \in N(C),  \  k \not =  {k''}\}.
\end{align}
To
each triplet  $ ( k, {k'}, {k''})  \in T(C)$, we uniquely
associate the angle $\theta(\bx_k,\bx_{k'},\bx_{k''})$ formed by the
 vectors 
$\bx_k-\bx_{k'}$ and $\bx_{k''}-\bx_{k'}$ and oriented clockwise.  Note that $\theta(\bx_k,\bx_{k'},\bx_{k''}) +  \theta(\bx_{k''},\bx_{k'},\bx_{k}) = 2\pi$. 
We use the notation $\bT(C)=\{(\bx_k,\bx_{k'},\bx_{k''}) \colon  \,
(k,k',k'') \in T(C)\}$ to indicate triplets of bonded points in the
configuration.

We call the  collection of all bonds of a configuration $C$
 the  {\it bond
  structure}  of $C$.  A  collection of pairwise  distinct  points $(\bx_1,\ldots, 
 \bx_m)$   in $C$   is called a \emph{path} if $ (  i, {i+1} )
  \in
N( C  )$ for $i=1,\ldots,m-1$. Moreover, we say $(\bx_1,\ldots,
\bx_m)$ is a \emph{cycle} if it is a path and $ (  m,  1
)  \in N( C  )$. By the Jordan curve theorem, each planar cycle  which is a closed loop without self-intersection   partitions
 its  support plane  into  an interior and an exterior.   We say
that a configuration  $C$ is {\it connected} if each two points of $C$ can be connected by a path.  In a similar fashion, we define
connected components of configurations.

  For $C \in \Zz^{2n}$, denoting by $B \subset \Rz^2$ the
 union of the bonds, i.e., the union of the  line segments between
 $\bx_k$ and $\bx_{k'}$ for  $( k, {k'}) \in N(C)$, we call the bounded connected components of the open set $\Rz^2 \setminus B$ the \emph{faces} of $C$. For each face $f$, we call the points  in $C$ on the boundary of the  face a \emph{cell}. More precisely, given $f$, we define $Z = \partial f \cap C$, and write $f(Z)$ to highlight that $f(Z)$ is the face associated to  the cell $Z$. The collection of all cells of $C$ is denoted by $\mathcal{Z}(C)$. Note that in general each point may be
 contained in different cells.

 Let now $\widetilde C=\{\tilde \bx_1, \dots
, \tilde \bx_n\} \in \Rz^{3n}$ be an {\it $\epsi$-perturbation} of the
configuration $C \in \Zz^{2n}$, namely,  $|\bx_k - \tilde \bx_k|
<\epsi$ for all $k\in \{1,...,n\}$. Owing to definition \eqref{eq: neighbor0}   and  $ 0 < \epsi \le  (\sqrt{2}-1)/4$,  we have that
$N(\widetilde C) = N(C)$ and  $T(\widetilde C) = T(C)$. In other
words, the topology of the  bond structure of $C$ and $\widetilde C$
are the same.    We can now pose the following.

\begin{definition}[Angle-rigid   $\Zz^2$
  configurations]\label{def: ar}
  We say that a configuration $C \in \Zz^{2n} $ is \emph{angle-rigid} 
  if there exists $\delta \in (0,\epsi)$ such that  any configuration $\widetilde C=\{\tilde \bx_1, \dots
, \tilde \bx_n\} \in \Rz^{3n}$ with $|\bx_k - \tilde \bx_k| <\delta$
for all $k=1, \dots, n$ and
 \begin{align*}
    &| \tilde \bx_k- \tilde  \bx_{k'}|=|\bx_k - \bx_{k'}|  \quad \forall
    (  k,  {k'})\in N(C)=N(\widetilde C),  \\
&\theta(\tilde \bx_k, \tilde \bx_{k'},
 \tilde \bx_{k''}) = \theta (\bx_k,\bx_{k'},\bx_{k''})\quad \forall
    ( k, {k'}, {k''})\in T(C)=T(\widetilde C), 
 \end{align*} 
 is congruent to $C$, namely, $|\tilde \bx_k -
  \tilde \bx_{k'}| = | \bx_k - \bx_{k'} |$ for all $
  k,\,k' =1,\dots,n$. 
\end{definition}

The above definition can be equivalently presented in more
mechanical terms.
Let us call a mapping $\bvarphi\colon   C
\to \Rz^{3n}$   \emph{angle-preserving}  if  $\tilde{C} := \lbrace \bvarphi(\bx_1),\ldots,\bvarphi(\bx_n)\rbrace$  is an $\epsi$-perturbation of $C$ and  
  \begin{align}
    &|\bvarphi(\bx_k)- \bvarphi(\bx_{k'})| = 1 \quad \forall
    ( k, {k'})\in N(C),\label{ar1}\\
&\theta(\bvarphi(\bx_k), \bvarphi(\bx_{k'}),
  \bvarphi(\bx_{k''})) = \theta (\bx_k,\bx_{k'},\bx_{k''})\quad \forall
    ( k, {k'}, {k''})\in T(C).\label{ar2}
  \end{align}
A configuration $C \in \Zz^{2n} $ is then angle-rigid
iff there exists $\delta \in (0,\epsi)$ such that all angle preserving
maps $\bvarphi$ with $|\bx_k - \bvarphi(\bx_k)|<\delta$ for all
$k=1,\dots,n$ are (restrictions to $C$ of)
isometries. 

 In the sequel, mappings $\bvarphi$ which are isometries (restricted to $C$)
will be called \emph{trivial}, otherwise \emph{nontrivial}. As already mentioned in the Introduction, angle-preserving mappings  that keep the points coplanar are trivial. Therefore, nontrivial angle-preserving mappings
necessarily have a nonplanar image.  The
focus of the  forthcoming  Sections \ref{sec: suff}--\ref{sec:char} is to discuss a  sufficient and a
conditional necessary condition  for  angle-rigidity 
which can be checked with limited  computational  effort.

\subsection{Relation with rigidity in graphs}\label{sec:discrete}
In  the  whole  paper, we  
use a notation adapted to our mechanical
application, by focusing on point configurations instead of the graph
structure of bonds.   This is in particular   tailored  to the check of strict
minimality of configurational energies, as discussed in Section
\ref{sec: appl} below.

Before moving on, we would like to mention, however, that our setting
is closely related to notions in graph-rigidity
theory \cite{Asimov,Laman,Lovasz}. In fact, the notation above corresponds to associating to each $\Zz^2$
configuration $C = \{\bx_k \}_{k=1}^n\in \Zz^{2n}$  the {\it unit-disk
   graph} $G=(V,E)$, namely, a graph with one vertex
 for each point in $C$, and with an edge between two vertices whenever
 the corresponding points in $C$ are at distance  $1$ \cite{Garam}.  In
 particular, the 
set of {\it edges} $E$ can be specified as
\begin{align*} 
E =\{( k, {k'}) \in  V\times V   \colon \,  |\bx_k-\bx_{k'}| =1  \}.
\end{align*}
 A {\it
   framework} is then a pair $(G,p)$, where $p : G \to \Rz^3$ is
 referred to as the
 {\it realization} of the graph $G$. We canonically associate to $C$ the framework $(G,p)$ given
 by   $p(k):= \bx_k$ for all $k\in V$. 
We may also indicate   {\it triplets} as 
$$T  = \{(k,{k'},{k''}) \in  V \times V \times V   \colon \, (k,{k'}), \,
({k'},{k''}) \in E,  \ k \not = {k''}\}$$
and associate to  
each triplet  $(k,{k'},{k''}) \in T$ the angle  $\theta(p(k), p({k'}),
  p({k''}))$  formed by the
 vectors   
$p(k)-p({k'})$ and $p({k''})-p({k'})$   and oriented clockwise.  We note that $E$ and $T$ coincide with $N(C)$ and $T(C)$ defined in \eqref{eq: neighbor0} and \eqref{eq: neighbor000}, respectively,  provided that we let  $V=\{1,...,n\}$.   


 A framework is said to be {\it rigid} if there
  exists $\delta>0$ such that all frameworks $(G,q)$ with
  $|p(k)-q(k)|<\delta$ for all $k \in V$ with the property
  \begin{align*}
    &|q(k)- q(k')| = |p(k)-p(k')|\quad \forall
      ( k, {k'})\in E
  \end{align*}
 are congruent to $(G,p)$, namely,
$|q(k)- q(k')| = |p(k)-p(k')|$ for all
$ ( k, {k'})\in V\times V$.  We refer the Reader  to \cite{Roth} for details about rigidity of frameworks.

Here, we add the restriction on the preservation of the angles and
call {\it angle-rigid} a framework $(G,p)$ if there exists
$\delta >0$ such that all frameworks $(G,q)$ with
  $|p(k)-q(k)|<\delta$ for all $k \in V$ with the properties
  \begin{align*}
    &|q(k)- q(k')| = |p(k)-p(k')|\quad \forall
      ( k, {k'})\in E,\\
&\theta(q(k), q({k'}),
  q({k''})) = \theta(p(k), p({k'}),
  p({k''}))\quad \forall
    (k,{k'},{k''})\in T, 
  \end{align*}
  are congruent to $(G,p)$.

  Given the above mentioned identifications,
  a $\Zz^2$ configuration $C$ is angle-rigid iff its associated
  framework $(G,p)$ is angle-rigid, so that these viewpoints are in
  fact equivalent. 
The graph theoretic notation may be more transparent in places. Still,
in the very specific situation considered   here,  namely, that of
$\Zz^2$ configurations, it seems to be adding an unnecessary degree of
abstraction. In fact, by  directly referring to graph-rigidity notions
would  complicate some of our
arguments, which are based on the direct manipulation of the underlying
configurations. We hence refrain from using graph-rigidity notions in
the following, focusing on point configurations instead.

 Let us mention that   rigidity  under angle preservation has recently attracted
attention. Applications arise from interacting systems of agents, as
in robotics and traffic flow \cite{Anderson,Oh}. Note that the term
{\it angle-rigid} has in fact already been used in \cite{Chen},
although in a different, purely planar context. See also
\cite{Buckley,Gangshang,Jing2,Jackson,Park,Saliola} for some other results in the plane, where angles come into
play.  In the specific case of molecular structures, the
celebrated {\it Molecular Conjecture} \cite{Tay} relates to the
possibility of investigating the rigidity of so-called {\it molecular
  graphs} \cite{Jackson2} by combinatorial methods.

Our notion of angle-rigidity is strictly related to that  of  
{\it body-and-hinge frameworks} \cite{Katoh}. These are collections of
$d$-dimensional rigid bodies connected by {\it hinges}, where a hinge
is identified with a $(d-2)$-dimensional affine subspace. Continuous motions of
the body-and-hinge framework are allowed under the constraints that
distances between two connected bodies are preserved and that their
relative motion is a rotation around the common
hinge. A body-and-hinge framework is said to be {\it rigid} if every such
admissible motion is an isometry. To each $\Zz^2$
configuration $\{\bx_1,\dots,\bx_n\}$ one can associate a
body-and-hinge framework by considering a small 3d sphere centered at
each $\{\bx_1,\dots,\bx_n\}$ and a hinge for each bond. Under this
identification, a $\Zz^2$
configuration is angle-rigid iff its associated   body-and-hinge
framework is rigid. Note however that the two notions are not
equivalent, for a  body-and-hinge
framework made of two hinged rigid bodies which are not rotationally
symmetric around the hinge is not rigid, whereas each $\Zz^2$
connected configuration made of two points is
angle-rigid. Combinatorial characterizations of the {\it infinitesimal} rigidity for
$d$-dimensional body-and-hinge frameworks are indeed available
\cite{Katoh}. These pave the ground for investigating the
angle-rigidity of molecules by well-developed  tree  algorithms on the
underlying graphs.

Our setting is here more restricted, for it exclusively applies to the
special class of $\Zz^2$ configurations.  Within this specific
setting, we are able to present a characterization of framework
rigidity. In comparison with \cite{Katoh}, our characterization has
the advantage of 
making the check of angle-rigidity straightforward, for the 
ensuing non-angle-rigidity sufficient  condition  \eqref{suff} is elementarily
implemented.  We expect that similar techniques can be applied to
other special classes, possibly including unit disk graphs generated
by subsets of the triangular or the hexagonal lattice.

\section{Strict  local minimality and angle-rigidity}\label{sec: appl}

 Before moving to the investigation of angle-rigidity for $\Zz^2$
configurations,  we comment here on the main application of our
theory, namely the characterization of locally stable point
configurations.

Assume to be given a {\it configurational energy} $\cE\colon \Rz^{3n}
\to  [0,+\infty]$ associating to each configuration  $ C  = \{\bx_k
\}_{k=1}^n\in \Rz^{3n}$ the scalar value 
 $\cE(C)$ given by
\begin{equation}\cE(  C  ) = \cE_2(  C  ) + \cE_3(  C  ) := \frac12 \sum_{N(C)} \vl(|\bx_k{-}\bx_{k'}|) + \frac12
\sum_{T(C)} \vt(\theta(\bx_k,\bx_{k'},\bx_{k''})).\label{eq:calE}
\end{equation}
The configurational energy is the sum of a {\it two-body} term $\cE_2$, depending solely on
the mutual distance of the points,
and a {\it three-body} contribution $\cE_3$ depending on angles instead. The two-body {interaction density}   $v_2\colon[0,\infty) \to
 [0,+\infty]$ is assumed to be strictly minimized at the
reference distance $1$  with $v_2(1)=0$. 
The three-body interaction energy $\cE_3$ is modulated  via the {three-body} interaction density $v_3 \colon
[0,2\pi] \to [0,\infty)$ which we assume to be symmetric around $\pi$
and  to  attain  its minimal value $0$ just at $\pi/2$, $\pi$, and
$3\pi/2$.   Note that the factors $1/2$ account for double counting
neighbors in $N(C)$ and triplets in $T(C)$.  Without claiming completeness, the Reader
is referred to the seminal  papers  \cite{Brenner,Stillinger-Weber85,Tersoff}
for the discussion of empirical potentials including three-body
contributions and to \cite{E, Farmer,Theil15, Mainini14,Mainini-Stefanelli12}
for some related mathematical crystallization results. The Reader is also
referred to \cite{Jacobs,Katoh,Whiteley2} for a collection of graph-theoretical results
of relevance in Molecular Mechanics.

 We   stress  that all $\mathbb{Z}^2$ configurations $C$ are global
 minimizers of the energy, as  one has $\cE(C)=0$.  We then ask ourselves if
 these configurations are  \emph{strict} local minimizers or not,
 where local minimality is understood with respect to the
 Euclidean norm in $\mathbb{R}^{3n}$.  Local minimality is indeed
 crucial in connection with applications, for it formalizes the
 concept  of  {\it stability} of discrete structures. 
 The relation of strict local minimality and angle-rigidity is given by the following.

\begin{proposition}[Angle-rigidity = strict local minimality]\label{prop:a}
  A  configuration is a strict local minimizer for the energy $\cE$ 
  up to isometries iff
  it is angle-rigid.
\end{proposition}

\begin{proof}
Let $C=\{\bx_k\}_{k=1}^n \in \Zz^{2n}$ be angle-rigid with given
$\delta \in (0,\epsi)$ and let 
$ \widetilde
C =\{\tilde \bx_1, \dots
, \tilde \bx_n\} \in \Rz^{3n}$ with $|\bx_k - \tilde \bx_k| <\delta$
for all $k=1, \dots, n$ such that $C$ and $\widetilde C$ are not
congruent. Hence, there exists either $( k, {k'})
\in N( C)$ such that $| \tilde \bx_k- \tilde  \bx_{k'}|\not
=1$ or there exists $( k,  {k'},
  {k''})\in T(\widetilde C)$ such that $\theta(\tilde \bx_k, \tilde \bx_{k'},
 \tilde \bx_{k''}) \not \in \{\pi/2,\pi,3\pi/2\}$ (or both). In all
 cases,  $ \cE(\widetilde C)  >0 = \cE(C)$,  which proves that $C$
is a strict local minimizer of $\cE$ in a $\delta$-neighborhood, up to isometries.  

Let now $C$ be not angle-rigid.  Then, for all $\delta<\epsi$
there exists $ \widetilde
C_\delta =\{\tilde \bx_1, \dots
, \tilde \bx_n\} \in \Rz^{3n}$ with $|\bx_k - \tilde \bx_k| <\delta$
for all $k=1, \dots, n$ such that $C$ and $\widetilde C_\delta$ are not
 congruent,  $| \tilde \bx_k- \tilde  \bx_{k'}|=1$ for all $(k,{k'})
\in N(\widetilde C_\delta)$,  and $\theta(\tilde \bx_k, \tilde \bx_{k'},
 \tilde \bx_{k''}) \in \{\pi/2,\pi,3\pi/2\}$ for all $(k,{k'},
  {k''})\in T(\widetilde C_\delta)$. Hence,    $\cE(\widetilde C_\delta) =
\cE(C)=0$  and   $C$ is not a strict local minimizer. 
\end{proof}

 We conclude this section by commenting on the relevance of this
stability result with respect to applications. Our setting is admittedly restricted, for the
configurational energy $\mathcal E$ solely features nearest-neighbors
interactions, together with some aspects of next-to-nearest-neighbors
interactions. In particular, long-range interactions are not
considered, which is surely a limitation. Note however that our energy is  tailored to
the modeling of ionic and covalent bonding at very low temperature. In
this setting,  the effect of long-range interactions on the structure
is usually negligible. This is for instance the case in honeycomb
carbon (sp$^2$-hybridization), where bonds range from 1.39$\pm$0.02,\AA\
for benzene to 1.42\,\AA\ for graphene, thus
differring by  2\%.

In addition, including long-range interactions in the model   
would contribute another source of rigidity, due to long-range
effects. We believe
that it would make the understanding of the interplay between geometry and
rigidity less transparent.
Let us moreover mention that rigorous crystallization results taking long-range
    interactions into account are scarce and usually very
    specific with respect to the choice of the interaction potentials. To the best of our knowledge, no long-range
    crystallization result for 2d cubic systems at zero temperature is
    available. On the contrary, such crystallization result is
    available in our nearest- and next-to-nearest-neighbors setting \cite{Mainini14}.

The choice of investigating $\Zz^2$ in 3d is also very specific, and
not directly related to any specific 2d atomic system. Our choice is here
motivated by simplicity. We expect that similar results could be
obtained for other crystalline
2d geometries, as well, including for instance the regular triangular lattice, the
regular hexagonal lattice, and the general Bravais lattice. This
would however call for some substantial notational and
technical changes, due to the distinct
symmetries of the 2d structures. 

Note again that the energy is degenerate, for $\mathcal E=0$ on {\it
  any} configuration in $\Zz^2$. This crude idealization is motivated
by our interest in the phenomenon of conformational isomerism \cite{Allinger}, where
bond-rotation phenomena occur. Here, energy barriers between
alternative configurations, although being not zero,
are usually small ($10^0\sim10^1$ kcal/mol) in comparison with ionic
bonding energy ($10^2\sim10^3$ kcal/mol) \cite{Liu}.

\section{Sufficient condition:  Combinatorial characterization of foldings}\label{sec: suff}

All disconnected  configurations  are obviously   not angle-rigid.  We
will hence concentrate on connected configurations from now on. 
Checking angle-rigidity for connected configurations  directly from the
definition may be very demanding. We are  thus  interested in reducing
the complexity of the problem by presenting a more tractable condition
implying  non-angle-rigidity.  Recall that we call axes the
sets  
$\{i\} \times \Zz\subset \Zz^2$ or $\Zz \times \{i\}\subset \Zz^2$, for
some  $i\in \Zz$.   In this section we prove the following sufficient condition.   
\begin{proposition}[Sufficient condition]\label{prop: suff}
Let $C \in \Zz^{2n}$ be connected and fulfill the following condition:
\begin{align}
  &\text{there exists an axis $A$ such that, by removing from $C$
  all points on $A$ having}\nonumber\\
&\text{at most one neighbor off
  $A$, the resulting configuration is disconnected.}\label{suff}
\end{align}
Then, $C$ is   not angle-rigid.  
\end{proposition}
 Condition \eqref{suff} is illustrated in Figure \ref{fig:added}. 
 Checking  condition \eqref{suff} for some general configuration  $C \in
\Zz^{2n}$  can be  accomplished very efficiently. Indeed, one has to
consider all axes intersecting $C$ (at most  $2n$),  remove points
according to \eqref{suff}, and check for connectedness of the
resulting configuration.  
\begin{figure}[ht]
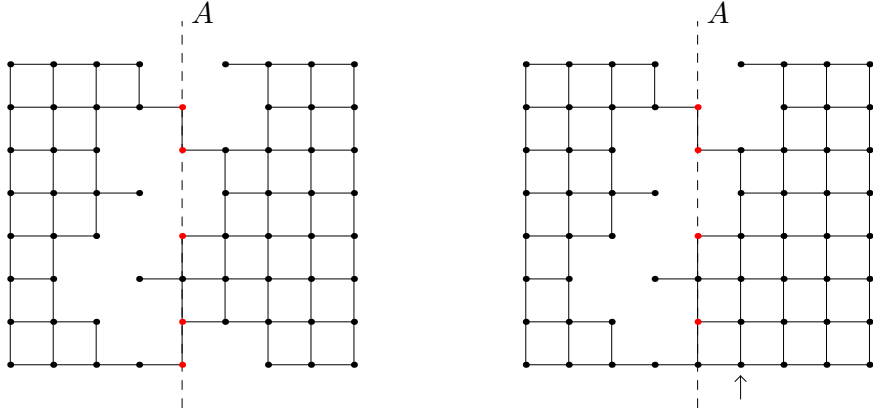

\centering
\pgfdeclareimage[width=115mm]{added}{added2}  
\pgfuseimage{added} 
\caption{ Illustration of condition \eqref{suff}. In both
  configurations, points on the axis $A$ having at most one
  neighbor off $A$ are highlighted. Upon removing them, the
  left configuration  becomes disconnected (i.e., condition
  \eqref{suff} holds) whereas the right configuration is still connected (condition
  \eqref{suff}  does not hold, respectively). Note that the two
  configurations differ by the  point  indicated by the arrow.}
\label{fig:added}
\end{figure}

 The proof of Proposition \ref{prop: suff} relies on an equivalent geometric characterization of condition \eqref{suff} in terms of a specific class of configurations, so-called {\it foldings} $\calF$. 

\begin{definition}[Folding]\label{def:s}
  We say that a configuration $C  \in   \Zz^{2n} $ is  a  \emph{folding},
  and we write $C\in \calF$, if there exists an axis $A$ and a
  subconfiguration $\widetilde C \subset C$ such that both $\widetilde
  C\setminus A$ and $(C \setminus \widetilde C)\setminus A$ are not empty and, letting  $R_\rho$  be
  a rotation about $A$ of amplitude  $\rho$  with  $ \rho  \in
  (0,\pi/2)$  arbitrarily   small,  the map   $\bvarphi\colon  C \to \Rz^{3n}$  defined as 
$$\bvarphi(\bx) = 
\left\{
  \begin{array}{ll}
    R_\rho(\bx)   \quad&\text{if} \ \bx \in \widetilde C\\
\bx &\text{if} \ \bx \in C\setminus \widetilde C
  \end{array}
\right.
$$
 is angle-preserving, i.e., fulfills \eqref{ar1}--\eqref{ar2}. 
\end{definition}

 We have the following characterization.  

\begin{proposition}[Characterization]\label{prop: F} $C
  \in \calF$ $\Leftrightarrow$ \eqref{suff} holds.
\end{proposition}

In other words,  condition \eqref{suff} can be understood as a combinatorial characterization of $\mathcal{F}$. Of course, all foldings are   not angle-rigid   since  the map
 $\bvarphi$  from Definition~\ref{def:s}   is    angle-preserving  and nontrivial.  Consequently, Proposition \ref{prop: F} immediately implies Proposition \ref{prop: suff}. We now prove Proposition \ref{prop: F}.

\begin{proof}
 \noindent \emph{Step 1: Implication \eqref{suff} $\Rightarrow$ $C
  \in \calF$}.   Let $C$ be connected and fulfill \eqref{suff}. Call $\widetilde C$ the
  configuration obtained by removing all points of $C$ belonging to
  $A$ and having at most one neighbor off $A$. Fix  a connected component  $C_1\subset \widetilde
  C $   of $ \widetilde
  C $.  Then,   $C_1$ and $\widetilde C\setminus C_1$ are 
  disconnected. Let  $R_\rho$   represent a rotation
  around the axis $A$ of amplitude  $\rho$,   with $ \rho \in
  (0,\pi/2)$  small,  and define the
  mapping  $\bvarphi\colon C \to \Rz^{3n}$   as
  \begin{align*}
    \bvarphi(\bx)=
\left\{
    \begin{array}{ll}
      R_\rho  (\bx)  \quad&\text{if} \ \ \bx \in C_1\\
\bx & \text{if} \ \ \bx \in C\setminus C_1.
    \end{array}
\right.
  \end{align*}
As $C_1$ and $\widetilde C\setminus C_1$ are 
  disconnected, the map $\bvarphi$ fulfills \eqref{ar1}--\eqref{ar2} when
  restricted to $ \widetilde C$.  

We aim at showing that, by reintegrating the points
  in $C\setminus \widetilde C$ which were removed under condition \eqref{suff}, relations
  \eqref{ar1}--\eqref{ar2} still hold.  This will show that
$\bvarphi$  is angle-preserving, and therefore $C \in \mathcal{F}$.   We proceed sequentially,
  by adding one point at a time. To this aim, assume to be given a configuration $\haz C $ with 
$\widetilde C \subset \haz C \subset C$ such that \eqref{ar1}--\eqref{ar2} hold for
  $\haz C$ but $\haz C \not = C$. Without loss of generality, we let  $A=\Zz \times \{0\}$ and 
$(0,0)\in C\setminus \haz C$. We now check that, by adding
the point $(0,0)$ to $\haz C$, properties \eqref{ar1}--\eqref{ar2}
still hold true. To prove this, we
need to check that the new bonds and   new  angles that have to be included 
by considering  the extra point $(0,0)$ are still invariant under
  $\bvarphi$.   Such new bonds and angles are determined by
the  pairs 
and the triplets 
$$\bN':=\bN(\haz C \cup (0,0))\setminus \bN(\haz C) \ \ \text{and}\ \
\bT':=\bT(\haz C \cup (0,0))\setminus \bT(\haz C),$$
respectively. (We refer to Section \ref{se:nota} for the definition of $\bN$ and $\bT$.) 

 In view of condition \eqref{suff},   as $(0,0) \not \in \widetilde C$, we have that at least one  of the points $(0,1)$ and $(0,-1)$  does not belong to  $C$, and thus does not belong to  $\haz C$. In case both do not
belong to $\haz C$,  the new bonds  $\bN'$ 
are necessarily aligned with $A$, namely 
\begin{equation}
{ \bN'  } \subset \big\{((0,0),(1,0)),\, ((1,0),(0,0)), \, ((0,0),(-1,0)),\,
((-1,0),(0,0)) \big\}.\label{aligned}
\end{equation}
Since   $\bvarphi$  is the identity on $A$, we have that $|\bvarphi(\bx_k) -
\bvarphi(\bx_{k'})|=1$ for all  $(\bx_k,\bx_{k'})\in \bN'$  as well. As
regards triplets, we have that 
\begin{equation}(\bx_k,\bx_{k'},\bx_{k''})\in  \bT'  \ \Rightarrow \ \text{at least two
  out of $ \lbrace \bx_k,  \bx_{k'}, \bx_{k''}\rbrace$ belong to $A$}. \label{aligned2}
\end{equation}
As  $\bvarphi$  is the identity on $A$ and either the identity or a
rotation about $A$ out of $A$, one has that $\theta (\bvarphi(\bx_k),
\bvarphi(\bx_{k'}), \bvarphi(\bx_{k''}))=\theta
(\bx_k,\bx_{k'},\bx_{k''})$ for all $(\bx_k,\bx_{k'},\bx_{k''})\in
 \bT'$. In particular,   $\bvarphi$  fulfills  properties
\eqref{ar1}--\eqref{ar2}  on   $\haz C\cup (0,0)$.  

Let us now consider the case $  (0,1)  \in \haz C$ (the remaining case
$ (0,-1)  \in \haz C$ can be treated analogously). Here, all bonds in
 $\bN'$  are either aligned with $A$ as in \eqref{aligned}, or
orthogonal to $A$. In both cases, as  $\bvarphi$  is the identity on $A$ and either the identity or a
rotation about $A$ out of $A$, one has that  $|\bvarphi(\bx_k) -
\bvarphi(\bx_{k'})|=1$ for all $(\bx_k,\bx_{k'})\in  \bN'$ as well. Those
triplets in  $\bT'$  containing at least two points on $A$ can be treated
as before, see \eqref{aligned2}. We are just left to consider the
triplets of the form
$$ \big((0,0),(0,1),(0,2)\big), \ \big((0,0),(0,1),(1,1)\big), \
\big((0,0),(0,1),(-1,1)\big),$$
(and permutations $(\bx_k,\bx_{k'},\bx_{k''}) \mapsto
(\bx_{k''},\bx_{k'},\bx_{k})$), if at all included in  $\bT'$.   Let us
assume that 
$((0,0),(0,1),(1,1))\in  \bT'$. Then, $((0,1),(1,1))\in  \bN(C)$, implying that
$(0,1)$ and $(1,1)$ belong to the same connected component of
$\widetilde C$. In case $(0,1), \, (1,1) \in \widetilde C \setminus
C_1$ we have that  $\bvarphi$  is the identity on $(0,0)$, $(0,1)$, and
$(1,1)$. In case $(0,1), \, (1,1) \in 
C_1$ we have that  $\bvarphi$  is a rotation  about $A$  on $(0,0)$, $(0,1)$, and
$(1,1)$. In both cases,  $\bvarphi$  is an isometry on $(0,0)$, $(0,1)$, and
$(1,1)$, implying that 
$$\theta (\bvarphi(0,0),
\bvarphi(0,1), \bvarphi(1,1))=\theta
((0,0),(0,1),(1,1)).$$ 
An analogous argument applies to
$((0,0),(0,1),(0,2)) $ and
$((0,0),(0,1),(-1,1)) $.  Thus,  we have
proved that  $\theta (\bvarphi(\bx_k),
\bvarphi(\bx_{k'}), \bvarphi(\bx_{k''}))=\theta
(\bx_k,\bx_{k'},\bx_{k''})$ for all $(\bx_k,\bx_{k'},\bx_{k''})\in
 \bT'$. Eventually, also in case $  (0,1)  \in \haz C$,  the map
 $\bvarphi$  fulfills  properties \eqref{ar1}--\eqref{ar2}  
 on    $\haz C\cup (0,0)$. 

 In summary, we have shown  that   the map
 $\bvarphi$  fulfills  properties \eqref{ar1}--\eqref{ar2}  
on  $\haz C\cup (0,0)\subset C$.   In case $\haz C\cup (0,0)\not = C$,
we repeat the argument by adding a point from  $C\setminus (\haz C\cup
(0,0))$  to $\haz C\cup (0,0)$,  making sure that $\bvarphi$
preserves  properties
\eqref{ar1}--\eqref{ar2}. Since the number of points in $C$ is finite,
this procedure eventually terminates, proving that  $\bvarphi$
 fulfills
\eqref{ar1}--\eqref{ar2}  on $C$.   This shows that
$\bvarphi$  is angle-preserving, and therefore $C \in \mathcal{F}$.

   \noindent \emph{Step 2: Implication $C
  \in \calF$  $\Rightarrow$  \eqref{suff}.}   Consider  $C\in \calF$. Let  $A$ and $\widetilde C$ be the
axis and the subconfiguration from Definition~\ref{def:s}. Without
loss of generality, possibly by
changing coordinates,  we can assume that  $A=\Zz\times \{0\}$.

 The  mapping  $\bvarphi$   from Definition~\ref{def:s}  does not preserve the length of bonds  which  connect
points in $\widetilde C$ and $C\setminus \widetilde C$  and do not intersect the axis   $A$. This implies that $\widetilde C$ and $C\setminus \widetilde C$ are 
connected just through bonds intersecting   $A$.

We would like to prove that condition \eqref{suff} holds with this
same axis $A$. In order to do this, we shall prove that by removing from $C$
  all points on $A$ having at most one neighbor off
  $A$, the resulting configuration $\haz C$ is disconnected.

Assume by contradiction that $\haz C$ is connected. The two sets  $\widetilde
  C\setminus A$ and $(C \setminus \widetilde C)\setminus A$ are nonempty and disjoint. Take a path  in  $\haz C$   connecting a point in $\widetilde
  C\setminus A$ to a point in $(C \setminus \widetilde C)\setminus
  A$. As we have already  seen  that such path should contain bonds
   intersecting  $A$, with no loss of generality, again by possibly redefining coordinates,
  we can find a path of the form
$$\{(0,1), (0,0), \dots, (k,0),( k,  \pm 1)\}\subset \haz C$$
where $k\in  \Nz  \cup \lbrace 0 \rbrace $ is given, $(0,1)\in \widetilde C$ and $( k, \pm 1) 
\in C \setminus \widetilde C$  (here, it is intended that either
$(k,1)$ or $(k,-1)$ belong to the path,  and that for $k=0$ it is necessarily $(0,-1)$).  As all points $(0,0), \dots, (k,0)$
belong to $\haz C$ after the removal described in \eqref{suff}, one has that all points
$ (0,\pm 1), \dots, (k,\pm 1) $ belong to $C$ as well.  We first consider the case that $(k,1)\in C \setminus \widetilde C$ is contained in the path. Then,  as $(0,1)\in \widetilde
C$, we find $0 \le l \le k-1$ such that $ (l, 1) \in  \widetilde
C$ and $(l+1,1) \in  C \setminus \widetilde C$.  Hence, we have that
$${|\bvarphi  (l,1)   -\bvarphi (l+1,1) | = |  R_\rho(l,1)  - (l+1,1)| >1}$$
contradicting \eqref{ar1}.   On the contrary,  suppose that $(k,-1)\in C \setminus \widetilde C$ is contained in the path. If $(k,1) \in C \setminus \widetilde C$, we find a contradiction as before. Thus, we can assume that $(k,1) \in \widetilde C$. But then
$${\theta\big(\bvarphi  (k,1), \bvarphi  (k,0), \bvarphi  (k,-1)\big) =  \theta\big(R_\rho(k,1),  (k,0), (k,-1)\big) \neq \pi}$$
contradicting \eqref{ar2}.  We conclude that $\haz C$ is necessarily disconnected
and  thus  condition  \eqref{suff} holds.
\end{proof}

\section{Necessary condition: shear-resistant configurations $\calS_k$}\label{sec: cond}

 In the previous section, we have seen that the sufficient condition \eqref{suff} is also necessary for non-angle-rigidity when restricting to the subclass of foldings. Let us point out, however, that   \eqref{suff} is not
necessary in general,  as the example in Figure
\ref{fig:cookie} shows. 
  
\begin{figure}[ht]
\centering
\pgfdeclareimage[width=25mm]{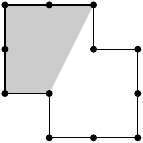}{cookie}  
\pgfuseimage{cookie} 
\caption{ A  not  angle-rigid  configuration not fulfilling
  \eqref{suff}. A nontrivial  angle-preserving  mapping 
  $\bvarphi$   is given by  $\bvarphi(\bx)=\bx$   at  the
boundary of the shaded region and $\bvarphi(\bx)=\bx+ (0,1-\cos
t,\sin t)  $ elsewhere,  for  $t>0$  small.  
}
\label{fig:cookie}
\end{figure}

  The remainder of the paper is   devoted to  another class 
of configurations whose angle-rigidity can be characterized via
\eqref{suff}. These will be called $k$-{\it shear-resistant}, in
coordination with an integer $k\in \Nz_0:=\Nz\cup\{0\}$, and will be
denoted by $\calS_k$. 

Recall that, to  each connected configuration $C
\in \Zz^{2n}$, we can associate its bond  structure.  In particular, from  Section \ref{se:nota}  we recall the definition of the cells $Z \in \mathcal{Z}(C)$ of a configuration $C$, and the corresponding faces $f(Z)$.  In the following, we say that a bond is \emph{acyclic} if it is not contained in any simple cycle of the bond graph, i.e., in a closed loop without self-intersection.

We start by noting that acyclic bonds of a cell may  influence 
angle-rigidity.   More precisely, a configuration with acyclic bonds
may    not be  angle-rigid  but become
angle-rigid, when all acyclic bonds are removed, see the examples in
Figure \ref{fig:acyclic}.

\begin{figure}[h]
\centering
\pgfdeclareimage[width=135mm]{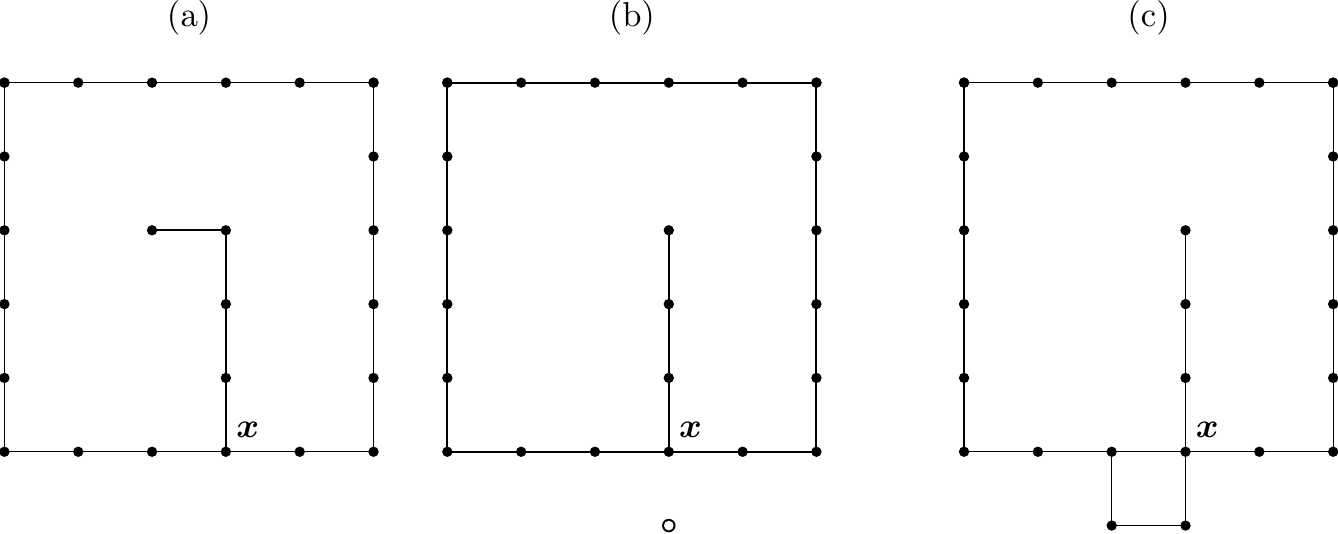}{acyclic}  
\pgfuseimage{acyclic} 
\caption{Illustration of   not angle-rigid  configurations. Configurations (a)-(b) get angle-rigid
  when acyclic bonds within the cells are removed, whereas   
  configurations (c)-(d) remain    not angle-rigid upon removal of
  the acyclic bonds. }
\label{fig:acyclic}
\end{figure}

In fact, the influence of acyclic bonds on  angle-rigidity 
can be completely characterized in terms of the class $\calF$, i.e.,
by condition \eqref{suff}.  To this end, we consider a single connected components of acyclic bonds intersecting a simple cycle in $\bx$.   With reference to  Figure   \ref{fig:acyclic},
we have that:  (a) If  there does not exist an axis containing all acyclic bonds, then condition \eqref{suff} holds for the vertical axis containing $\bx$,  and the configuration is
not angle-rigid; (b) If acyclic bonds are aligned on an axis and $\bx$
has only one bond on this axis,  condition  \eqref{suff}
holds for the horizontal axis containing the atom $\bx$; (c) If $\bx$
has two bonds on the axis,  the  angle-rigidity of the  configuration
is not affected by removing the acyclic bonds.  Indeed, in view of
\eqref{ar2}, any angle-preserving mapping on the configuration
obtained by removing the acyclic bonds can be uniquely extended to the
original configuration; (d) If the acyclic bond connects two cycles,
then condition \eqref{suff} holds for the horizontal axis containing
$\boldsymbol y$,  and the configuration is
not angle-rigid.      Note that  (a)--(d)  exhaust all possibilities  (up to multiple connected components of acyclic bonds). 
As we eventually aim at assessing if a configuration lies in $\mathcal{F} \cup \mathcal{S}_k$, we hence simplify the presentation from now on and assume with no further mention that  there are no  acyclic bonds.

Let $d$ denote the classical  Manhattan  distance  in
$\Zz^2$, namely  $$d (\bx,\by) =
|\bx_1-\by_1| + |\bx_2-\by_2|$$ for $ \bx =  (x_1,x_2),\,\by =
(y_1,y_2) \in \mathbb{Z}^2$.     Moreover, 
let $d(\bx,A) = \inf_{\by \in A}
d (\bx,\by)$ for $\bx \in \mathbb{Z}^2$ and $A \subset
\mathbb{Z}^2$   nonempty.

Given  the cell  $Z \in \mathcal{Z}(C)$ and $k \in
\Nz_0$, we define the corresponding
\emph{$k$-cell} $Z^{(k)}$    by  
\begin{align*}
Z^{(k)}:=  \{ \bx \in C \ | \ d( \bx,  Z)\leq k\}.
\end{align*}

 For each $k \in \Nz_0$, the collection of all $k$-cells
$Z^{(k)}$ for $Z \in \mathcal{Z}(C)$ is denoted by
$\mathcal{Z}^{(k)}(C)$.    Note that $Z=Z^{(0)}$, so that  
cells are also called  $0$-cells  in the following.

%

Four points $\bx_1, \bx_2, \bx_3,\bx_4 \in \Zz^2$ are said
to form a {\it paraxial} rectangle if  they are vertices of a
rectangle in $\Zz^2$ with sides aligned to axes in $\Zz^2$: namely, if
$\bx_1, \bx_2, \bx_3,\bx_4$ can be obtained by translating via
$\boldsymbol z \in \Zz^2$ the points 
$$(0,0), \ (a,0), \ (0,b), \ (a,b) \quad \text{for some} \ a,\, b \in
\Nz  \quad \text{with $\min\lbrace a,b \rbrace =1$}.$$ 
 We are now ready to define $k$-shear-resistant configurations.

\begin{definition}[$k$-shear-resistant configurations]\label{def: shear}
Let $C \in \Zz^{2n}$ be connected and let $k \in  \Nz_0$.    

(i) We say that a cell $Z \in \mathcal{Z}(C)$ is \emph{$k$-shear-resistant} if, letting $Z^{(k)} \in \mathcal{Z}^{(k)}(C)$ be
the corresponding $k$-cell, each   angle-preserving   mapping $\bvarphi\colon
 Z^{(k)} \to \Rz^{3\#Z^{(k)}}$  has the following property: if two pairs
$(\bx_i,\bx_{i'}), (\bx_j,\bx_{j'}) \in  \bN(Z)$ of pairwise
distinct points form a paraxial rectangle whose interior lies
entirely in $f(Z)$, then $$\bvarphi(\bx_i), \bvarphi(\bx_{i'}),
\bvarphi(\bx_{j}), \bvarphi(\bx_{j'})$$ are  coplanar.   

(ii) We  say that the configuration $C$ is 
\emph{$k$-shear-resistant},   and we write $C \in \calS_k$,   if all   its   cells are $k$-shear-resistant.  
\end{definition}

 The above definition  implies that $\calS_k \subset
\calS_{k'}$ if $k\leq k'$,  for the set of  nontrivial 
angle-preserving  mappings  of $Z^{(k)}$  decreases when $k$ increases.  Moreover, every angle-rigid configuration consisting of $n$ points lies in $\calS_n$.   
In Figure \ref{fig:three} we present examples  and non-examples of $k$-shear-resistant configurations for various
values of $k$.  
\begin{figure}[ht]
\centering
\pgfdeclareimage[width=145mm]{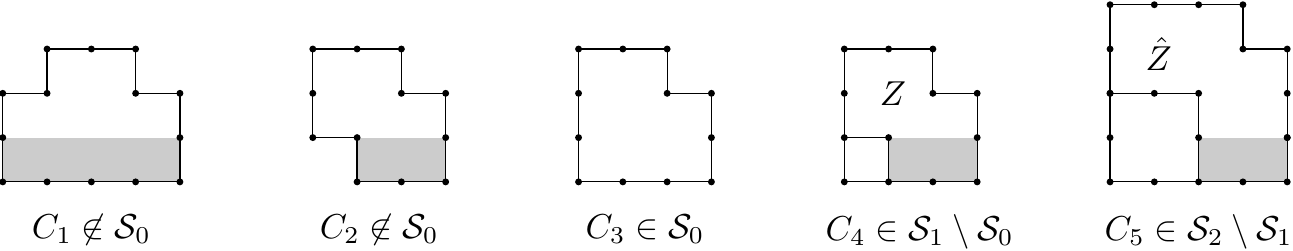}{three}  
\pgfuseimage{three} 
\caption{Examples of configurations  which are (not) shear-resistant,
  cf.~Section \ref{sec:char}  below.}
\label{fig:three}
\end{figure}

 We emphasize that, for a cell $Z$, the property of being $k$-shear resistant
depends \emph{not only on $Z$ itself but also} on the configuration $C$. In particular, the same cell may be
$k$-shear resistant as part of one configuration and not $k$-shear-resistant as part of another.   Following the discussion of Figure \ref{fig:cookie}, one has that
  the configuration   $C_2$   in Figure~\ref{fig:three}  is  not $0$-shear-resistant. Also $C_1$   in Figure~\ref{fig:three} is not  $0$-shear-resistant.  Indeed,  in  Figure~\ref{fig:cookie} we have described a   nontrivial angle-preserving
 mapping  of  $C_2$   making the marked paraxial rectangle
nonplanar. For $C_1$, we refer to the construction in the proof of Lemma \ref{lemma: 0-shear char} below, in particular to \eqref{eq: transformation}.   We will give a
complete characterization of $0$-shear-resistant configurations in Propositions
\ref{lemma: 0-shear char} and \ref{lemma: 0-shear char-gen}  later on, which will imply in particular that
 $C_3$  in Figure
\ref{fig:three} is $0$-shear-resistant. 

Cell $Z$ of configuration
 $C_4$  in Figure
\ref{fig:three} is again  $C_2$,  hence not $0$-shear-resistant. As a consequence,
 $C_4$  is not $0$-shear-resistant. 
On the
other hand, when considering $ C_4 = Z^{(1)}$ the marked paraxial
rectangle is  forced  to stay coplanar  for any angle-preserving  mapping.   This in particular implies that $ C_4  \in \calS_1$.

In the case of configuration  $C_5$  in Figure
\ref{fig:three}, one has that the top cell $\hat Z$ is not
$0$-shear-resistant. Any of its nontrivial angle-preserving
  mappings   cannot keep  the marked paraxial rectangle
coplanar.  Such
transformations are not precluded by considering $\hat Z^{(1)}$  since by passing from $\hat{Z}$ to  $\hat Z^{(1)}$ the configuration is only augmented by acyclic bonds. Thus,   $C_5$  is not $1$-shear-resistant as well. On the
other hand, $ C_5=  \hat Z^{(2)}$ cannot be
deformed  nontrivially preserving  the coplanarity of the marked paraxial
rectangle.  Hence, $C_5  \in \calS_2\setminus \calS_1$.

Configuration  $C_2$  in Figure \ref{fig:three} is an example of a
configuration which is not $k$-shear-resistant for any $k \in
\Nz_0$ (in fact it is the {\it smallest} configuration  in the sense of number of points which is  not in
$\calS_0$).  Indeed, we have that  $(C_2)^{(k)}=C_2 $  for all $k\in
\Nz$, as it consists of a single cell, so that no additional
angle-rigidity can follow by taking $k$ large. Note that the property
of not belonging to $\calS_k$ for any $k \in
\Nz_0$ does not necessarily imply
that the configuration is small. Indeed, Figure \ref{fig: shearing}
below shows that configurations which are  not  $k$-shear-resistant for any $k \in
\Nz_0$ can be arbitrarily large.  
\begin{figure}[h]
\centering
\begin{tikzpicture}[scale=0.45]

\draw[fill=blue!20, blue!20] (-15,1)--(11,1)--(11,4)--(-15,4)--cycle;
\draw[fill=blue!20, blue!20] (-15,1)--(-2,1)--(-2,-1)--(-15,-1)--cycle;
\draw[fill=red!20, red!20] (-14,-2)--(12,-2)--(12,-5)--(-14,-5)--cycle;
\draw[fill=red!20, red!20] (-1,-2)--(12,-2)--(12,0)--(-1,0)--cycle;

  \node at (4,1.75) {$f_2$};
    \node at (-1,-1) {$f_1$};

\draw[fill=black](0,0) circle(.07);
\draw[ultra thin](0,0)--++(1,0);
\draw[fill=black](1,0) circle(.07);
\draw[ultra thin](1,0)--++(0,1);
\draw[fill=black](3,0) circle(.07);
\draw[ultra thin](3,0)--++(0,1);
\draw[ultra thin](3,0)--++(1,0);
\draw[fill=black](4,0) circle(.07);
\draw[ultra thin](4,0)--++(1,0);
\draw[fill=black](5,0) circle(.07);
\draw[ultra thin](5,0)--++(0,1);
\draw[fill=black](7,0) circle(.07);
\draw[ultra thin](7,0)--++(1,0);
\draw[ultra thin](7,0)--++(0,1);
\draw[fill=black](8,0) circle(.07);
\draw[ultra thin](8,0)--++(1,0);
\draw[fill=black](9,0) circle(.07);
\draw[ultra thin](9,0)--++(0,1);
\draw[fill=black](11,0) circle(.07);
\draw[ultra thin](11,0)--++(1,0);
\draw[ultra thin](11,0)--++(0,1);
\draw[fill=black](12,0) circle(.07);

\draw[fill=black](1,1) circle(.07);
\draw[ultra thin](1,1)--++(1,0);
\draw[fill=black](2,1) circle(.07);
\draw[ultra thin](2,1)--++(1,0);
\draw[fill=black](3,1) circle(.07);
\draw[fill=black](5,1) circle(.07);
\draw[ultra thin](5,1)--++(1,0);
\draw[fill=black](6,1) circle(.07);
\draw[ultra thin](6,1)--++(1,0);
\draw[fill=black](7,1) circle(.07);
\draw[fill=black](9,1) circle(.07);
\draw[ultra thin](9,1)--++(1,0);
\draw[fill=black](10,1) circle(.07);
\draw[ultra thin](10,1)--++(1,0);
\draw[fill=black](11,1) circle(.07);
\draw[fill=black](-1,0) circle(.07);
\draw[ultra thin](-1,0)--++(1,0);
\draw[ultra thin](-1,0)--++(0,1);
\draw[fill=black](-3,0) circle(.07);
\draw[ultra thin](-3,0)--++(0,-1);
\draw[ultra thin](-3,0)--++(0,1);
\draw[fill=black](-7,0) circle(.07);
\draw[ultra thin](-7,0)--++(0,-1);
\draw[ultra thin](-7,0)--++(0,1);
\draw[fill=black](-11,0) circle(.07);
\draw[ultra thin](-11,0)--++(0,-1);
\draw[ultra thin](-11,0)--++(0,1);
\draw[fill=black](-15,0) circle(.07);
\draw[ultra thin](-15,0)--++(0,-1);
\draw[ultra thin](-15,0)--++(0,1);

\draw[fill=black](-15,-1) circle(.07);
\draw[ultra thin](-15,-1)--++(1,0);
\draw[fill=black](-14,-1) circle(.07);
\draw[ultra thin](-14,-1)--++(0,-1);
\draw[fill=black](-14,-2) circle(.07);
\draw[ultra thin](-14,-2)--++(1,0);

\foreach \j in {0,...,13}{
\draw[fill=black](-15 + \j,1) circle(.07);
\draw[ultra thin](-15+\j,1)--++(1,0);
\draw[ultra thin](-15+\j,1)--++(0,1);
}
\draw[fill=black](-1,1) circle(.07);

\foreach \j in {0,...,12}{
\draw[fill=black](-15 + \j,2) circle(.07);
\draw[ultra thin](-15+\j,2)--++(1,0);
\draw[ultra thin](-15+\j,2)--++(0,1);
}

\foreach \j in {0,...,12}{
\draw[fill=black](-15 + \j,3) circle(.07);
\draw[ultra thin](-15+\j,3)--++(1,0);
\draw[ultra thin](-15+\j,3)--++(0,1);
}

\foreach \j in {0,...,12}{
\draw[fill=black](-15 + \j,4) circle(.07);
\draw[ultra thin](-15+\j,4)--++(1,0);
}

\foreach \j in {0,...,12}{
\draw[fill=black](-3 + \j,3) circle(.07);
\draw[ultra thin](-3+\j,3)--++(1,0);
\draw[ultra thin](-3+\j,3)--++(0,1);
}
\draw[fill=black](10,3) circle(.07);
\draw[ultra thin](10,3)--++(0,1);

\foreach \j in {0,...,12}{
\draw[fill=black](-3 + \j,4) circle(.07);
\draw[ultra thin](-3+\j,4)--++(1,0);
}
\draw[fill=black](10,4) circle(.07);


\draw[fill=black](-2,2) circle(.07);
\draw[ultra thin](-2,2)--++(0,-1);
\draw[ultra thin](-2,2)--++(0,1);
\draw[fill=black](2,2) circle(.07);
\draw[ultra thin](2,2)--++(0,-1);
\draw[ultra thin](2,2)--++(0,1);
\draw[fill=black](6,2) circle(.07);
\draw[ultra thin](6,2)--++(0,-1);
\draw[ultra thin](6,2)--++(0,1);
\draw[fill=black](10,2) circle(.07);
\draw[ultra thin](10,2)--++(0,-1);
\draw[ultra thin](10,2)--++(0,1);
\draw[fill=black](0,-1) circle(.07);
\draw[ultra thin](0,-1)--++(0,-1);
\draw[ultra thin](0,-1)--++(0,1);
\draw[fill=black](4,-1) circle(.07);
\draw[ultra thin](4,-1)--++(0,-1);
\draw[ultra thin](4,-1)--++(0,1);
\draw[fill=black](8,-1) circle(.07);
\draw[ultra thin](8,-1)--++(0,-1);
\draw[ultra thin](8,-1)--++(0,1);
\draw[fill=black](12,-1) circle(.07);
\draw[ultra thin](12,-1)--++(0,-1);
\draw[ultra thin](12,-1)--++(0,1);

\begin{scope}[shift={(-13,-2)}]

\draw[fill=black](0,0) circle(.07);
\draw[ultra thin](0,0)--++(1,0);
\draw[fill=black](1,0) circle(.07);
\draw[ultra thin](1,0)--++(0,1);
\draw[fill=black](3,0) circle(.07);
\draw[ultra thin](3,0)--++(0,1);
\draw[ultra thin](3,0)--++(1,0);
\draw[fill=black](4,0) circle(.07);
\draw[ultra thin](4,0)--++(1,0);
\draw[fill=black](5,0) circle(.07);
\draw[ultra thin](5,0)--++(0,1);
\draw[fill=black](7,0) circle(.07);
\draw[ultra thin](7,0)--++(1,0);
\draw[ultra thin](7,0)--++(0,1);
\draw[fill=black](8,0) circle(.07);
\draw[ultra thin](8,0)--++(1,0);
\draw[fill=black](9,0) circle(.07);
\draw[ultra thin](9,0)--++(0,1);
\draw[fill=black](11,0) circle(.07);
\draw[ultra thin](11,0)--++(1,0);
\draw[ultra thin](11,0)--++(0,1);
\draw[fill=black](12,0) circle(.07);
\draw[ultra thin](12,0)--++(1,0);
\draw[fill=black](13,0) circle(.07);
\draw[fill=black](1,1) circle(.07);
\draw[ultra thin](1,1)--++(1,0);
\draw[fill=black](2,1) circle(.07);
\draw[ultra thin](2,1)--++(1,0);
\draw[fill=black](3,1) circle(.07);
\draw[fill=black](5,1) circle(.07);
\draw[ultra thin](5,1)--++(1,0);
\draw[fill=black](6,1) circle(.07);
\draw[ultra thin](6,1)--++(1,0);
\draw[fill=black](7,1) circle(.07);
\draw[fill=black](9,1) circle(.07);
\draw[ultra thin](9,1)--++(1,0);
\draw[fill=black](10,1) circle(.07);
\draw[ultra thin](10,1)--++(1,0);
\draw[fill=black](11,1) circle(.07);

\end{scope}


\draw[fill=black](-1,-3) circle(.07);
\draw[ultra thin](-1,-3)--++(0,-1);
\draw[ultra thin](-1,-3)--++(0,1);
\draw[fill=black](-5,-3) circle(.07);
\draw[ultra thin](-5,-3)--++(0,-1);
\draw[ultra thin](-5,-3)--++(0,1);
\draw[fill=black](-9,-3) circle(.07);
\draw[ultra thin](-9,-3)--++(0,-1);
\draw[ultra thin](-9,-3)--++(0,1);
\draw[fill=black](-13,-3) circle(.07);
\draw[ultra thin](-13,-3)--++(0,-1);
\draw[ultra thin](-13,-3)--++(0,1);

\foreach \j in {0,...,11}{
\draw[fill=black](\j,-2) circle(.07);
\draw[ultra thin](\j,-2)--++(1,0);
}
\draw[fill=black](12,-2) circle(.07);

\foreach \j in {0,...,12}{
\draw[fill=black](\j-1,-3) circle(.07);
\draw[ultra thin](\j-1,-3)--++(1,0);
\draw[ultra thin](\j-1,-3)--++(0,1);
}
\draw[fill=black](12,-3) circle(.07);
\draw[ultra thin](12,-3)--++(0,1);

\foreach \j in {0,...,12}{
\draw[fill=black](\j-1,-4) circle(.07);
\draw[ultra thin](\j-1,-4)--++(1,0);
\draw[ultra thin](\j-1,-4)--++(0,1);
}
\draw[fill=black](12,-4) circle(.07);
\draw[ultra thin](12,-4)--++(0,1);

\foreach \j in {0,...,12}{
\draw[fill=black](\j-1,-5) circle(.07);
\draw[ultra thin](\j-1,-5)--++(1,0);
\draw[ultra thin](\j-1,-5)--++(0,1);
}
\draw[fill=black](12,-5) circle(.07);
\draw[ultra thin](12,-5)--++(0,1);

\foreach \j in {0,...,12}{
\draw[fill=black](\j-13,-4) circle(.07);
\draw[ultra thin](\j-13,-4)--++(1,0);
}

\foreach \j in {0,...,12}{
\draw[fill=black](\j-13,-5) circle(.07);
\draw[ultra thin](\j-13,-5)--++(1,0);
\draw[ultra thin](\j-13,-5)--++(0,1);
}

\end{tikzpicture}
\caption{Example of a configuration  $C \not \in \calS_k$  for all $k \in \Nz_0$. Indeed, 
  the faces $f_2$ are
  $1$-shear-resistant, but $f_1$ is not $k$-shear-resistant for 
  all 
  $k \in \Nz_0$. (The corresponding nontrivial
   mapping  is the one indicated  by the  colors: leave  the points in the red region 
   in $\Rz^2 \times \lbrace 0 \rbrace$  and move  the points in the blue
   region  by the vector $(0,1-\cos t,\sin t)$ for any $t>0$ small.)     The  configuration extends the example provided in Figure~\ref{fig:cookie}.  }
\label{fig: shearing}
\end{figure}
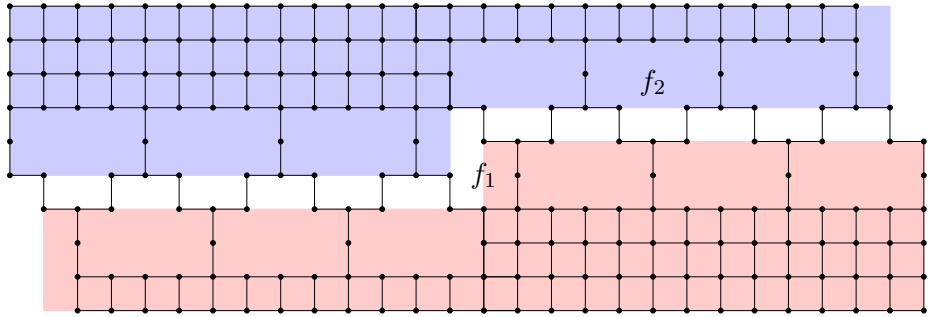

 The focus of this section is that of proving that
condition \eqref{suff} indeed characterizes angle-rigidity in
$\calS_k$. More precisely, we have the following.


%

\begin{proposition}[ Characterization,   shear-resistance]\label{prop: necessary}
 Consider $C  \in  \calS_k$ for some $k\in \Nz_0$. Then, $C$ is not angle-rigid iff condition \eqref{suff} holds. 
\end{proposition}

 Recall  that the $\calS_k$ assumption in
Proposition \ref{prop: necessary} cannot be removed. In fact, Figures~\ref{fig:cookie} and \ref{fig: shearing} give examples of
  not angle-rigid  configurations which do not fulfill \eqref{suff}.

\begin{proof}[Proof of Proposition \ref{prop: necessary}]
 If condition \eqref{suff} holds, then  $C$ is not angle-rigid by Proposition~\ref{prop: suff}. Thus, we only need to show the other implication.  
 We  start by pointing  out an aspect of   Definition~\ref{def:
  shear}. Given any cycle in the bond  structure  of a $k$-shear-resistant configuration $C$, and four different points   $\bx_i,\bx_{i'}, \bx_j,\bx_{j'}$ with    $(\bx_i,\bx_{i'}), (\bx_j,\bx_{j'}) \in  \bN(C)$ which form a paraxial rectangle whose interior is contained in the interior of the cycle, then 
\begin{align}\label{eq: coplanar}
\bvarphi(\bx_i), \bvarphi(\bx_{i'}), \bvarphi(\bx_{j}),
  \bvarphi(\bx_{j'})  \ \ \text{ are coplanar}
\end{align} 
for each  angle-preserving   mapping    $\bvarphi$.  In fact, if this were not the
case, we would find a cell in the interior of the cycle  and a corresponding paraxial rectangle whose images under $\bvarphi$ are not coplanar. Therefore, this cell would not be  $k$-shear-resistant. 

As $C$ is    not angle-rigid,   we find a nontrivial  angle-preserving
  mapping 
$\bvarphi$. In particular, the
points $\bvarphi(C) \subset \Rz^3$ are not coplanar  as  mappings 
$\bvarphi$ keeping the points coplanar are necessarily  trivial.
  As
$C$ is connected, we can hence find a path  $(\bx_1,\ldots, \bx_m)
\subset C$  such that   $\lbrace \bvarphi(\bx_i)
\rbrace_{i=2}^{m-1}$ are all contained in a line, but $\lbrace
\bvarphi(\bx_i) \rbrace_{i=1}^{m}$ are not coplanar. Without
restriction we can choose the number $m$ minimal, i.e., for all paths
 consisting of at most $(m-1)$ points  the images under  $\bvarphi$  are coplanar.
Moreover,  with no loss of generality, by possibly changing
coordinates we can  assume that  $\bvarphi(\bx_i) = \bx_i$  for all $i=1\ldots,m-1$,  $\lbrace \bx_i \rbrace_{i=2}^{m-1}$ are contained in the axis $\Zz \times \lbrace 0 \rbrace$,  $\bx_1 = (0,1)$, and 
\begin{align}\label{eq: last atom choice}
\bx_m \in \lbrace (m-3,1), (m-3,-1)\rbrace.
\end{align}
We also note that $\bvarphi(\bx_m) \in \lbrace m-3\rbrace \times
 \Rz^2$.  The minimality of $m$ implies  that all points $\lbrace \bx_i
\rbrace_{i=3}^{m-2}$ do not have neighbors off  the  axis $
\Zz \times \lbrace 0 \rbrace$. Indeed, such a neighbor, the axis $\Zz
\times \lbrace 0 \rbrace$, and one of the points $\bx_1$ or $\bx_m$
would not remain coplanar by the mapping  $\bvarphi$   and one could
hence find a shorter  noncoplanar  path.  

We now suppose by contradiction that  condition  \eqref{suff} does not hold, i.e.,
the configuration  $\widetilde C$   obtained by removing all
points of $C$ belonging to   $\Zz \times \lbrace 0 \rbrace$ and having
at most one neighbor off $\Zz \times \lbrace 0 \rbrace$  is still
connected. Therefore, there exists a path in $\widetilde C$, denoted
by $P$,   connecting $\bx_1$ and $\bx_m$. Note that the points
$\lbrace \bx_i \rbrace_{i=3}^{m-2}$ are not contained in $\widetilde
C$ as they do not have neighbors off axis.  Consequently, there exists
a cycle $S$  in $C$,  being a subset of the points $P \cup \lbrace \bx_i
\rbrace_{i=2}^{m-1}$.    We now argue that it is not restrictive to assume that $\bx_1$ and $\bx_m$ are contained in
the cycle $S$ or in the interior of the cycle $S$. In fact, at least one of the points $\bx_1 = (0,1)$ and $(0,-1)$ is contained in
the cycle $S$ or in the interior of the cycle $S$. If only $(0,-1)$ satisfies this property, this necessarily implies that $(0,0) \in P$ and therefore $(0,-1) \in C$ by \eqref{suff}. Thus,  we can simply replace   $\bx_1$ by $(0,-1)$ in the path $P$ observing that $\bvarphi(0,-1) = (0,-1)$ by property \eqref{ar2}. In a  similar fashion, by possibly reflecting $\bx_m$  along $\Rz \times \lbrace 0 \rbrace$ we can suppose that $\bx_m$ is contained in
the cycle $S$ or in the interior of the cycle $S$.   

We now distinguish
two cases: (a) $\bx_m  =  (m-3,1)$ and  (b) $\bx_m = (m-3,-1)$, cf.\ \eqref{eq: last atom choice}. 

In case (a), the points $\bx_1, \bx_2, \bx_{m-1},\bx_m$ form a paraxial rectangle whose interior is contained in the interior of $S$. (Recall that $\bx_1 = (0,1)$ and $\bx_2 = (0,0)$.) Then, \eqref{eq: coplanar} implies that $\bvarphi(\bx_1),\bvarphi(\bx_2),\bvarphi(\bx_{m-1}),\bvarphi(\bx_m)$ are coplanar. This, however, contradicts our choice of the path $(\bx_1,\ldots, \bx_m)$.

 Consider now  case (b). We first observe that the path $P$ in $\widetilde C$ connecting $\bx_1$ and $\bx_m$ needs to cross the axis $\Zz \times \lbrace 0 \rbrace$, i.e., we find $\bar{\bx} \in (\Zz  \times \lbrace 0 \rbrace) \cap P$. We observe that it is not restrictive to suppose that also 
\begin{align}\label{eq: upper/lower point}
\bar{\bx} + (0,1) \in P, \quad \quad \quad \bar{\bx} - (0,1) \in P.
\end{align} 
In fact, consider the smallest subpath  $(\hat{\bx}_1,\ldots,
\hat{\bx}_l)$     of $P \subset \widetilde{C}$ containing $\bar{\bx}$
such that $\hat{\bx}_1 \in \Zz \times \lbrace -1 \rbrace$, $\lbrace
\hat{\bx}_i \rbrace_{i=2}^{l-1} \subset \Zz \times \lbrace 0 \rbrace$,
and $\hat{\bx}_l \in \Zz \times \lbrace 1 \rbrace$. Then, by the
definition of $\widetilde C$  we have that $\lbrace \hat{\bx}_i +
(0,1) \rbrace_{i=2}^{l-1} \subset \widetilde C$ and  $\lbrace
\hat{\bx}_i - (0,1) \rbrace_{i=2}^{l-1} \subset \widetilde C$. This
means that we can connect $\hat{\bx}_1$ and $\hat{\bx}_l$ with a path
in $\widetilde C$ which contains only one point of the axis $\Zz
\times \lbrace 0 \rbrace$. By replacing the part $(\hat{\bx}_1,\ldots,
\hat{\bx}_l)$  of $P$ by this new path, we obtain the desired property
\eqref{eq: upper/lower point}.

 Define $\bar{\bx}' = \bar{\bx} -(0,1)$ and $\bar{\bx}'' = \bar{\bx}
+(0,1)$ for brevity, and recall that $\bar{\bx}',\bar{\bx}'' \in P$ by
\eqref{eq: upper/lower point}.  If the path $P$ crosses the  axis $\Zz \times \lbrace 0 \rbrace$ several times, we choose $\bar{\bx} \in (\Zz  \times \lbrace 0 \rbrace) \cap P$ to be the point closest to $\bx_2$. Then, we observe that $\bx_1,
\bx_2,  \bar{\bx}, \bar{\bx}''$ and $\bx_{m-1}, \bx_m,  \bar{\bx}, \bar{\bx}'$ form paraxial rectangles whose interior is contained in the interior of the cycle $S$.  In view of \eqref{ar2}, we get that the
points  $\bvarphi(\bar{\bx}    ), \bvarphi(\bar{\bx}'    ),
\bvarphi(\bar{\bx}''    )$ lie on a line. 
Moreover, as $\bx_1,
\bx_2,  \bar{\bx}, \bar{\bx}''$ form a paraxial rectangle whose
interior is contained in the interior of the cycle $S$, we find that
$\bvarphi(\bx_1), \bvarphi(\bx_2), \bvarphi(\bar{\bx}),
\bvarphi(\bar{\bx}'')$ are  also  coplanar by \eqref{eq: coplanar}. In a similar fashion, $\bx_{m-1}, \bx_m,  \bar{\bx}, \bar{\bx}'$ form a paraxial rectangle  contained in the interior of cycle $S$, and thus $\bvarphi(\bx_{m-1}), \bvarphi(\bx_m), \bvarphi(\bar{\bx}), \bvarphi(\bar{\bx}')$ are coplanar. This entails that also $\lbrace \bvarphi(\bx_i) \rbrace_{i=1}^m$ are coplanar which clearly contradicts our choice of the path $(\bx_1,\ldots, \bx_m)$. This concludes the proof.   
\end{proof}

%
%
%
%
%
%
%
%
%
%
%
%
%
%
%

 We close this discussion by summarizing the results of Sections
\ref{sec: suff}--\ref{sec: cond} in the
following statement.

\begin{corollary}[Necessary and sufficient condition]\label{cor} For  $C \in
  \calF  \cup  \calS_k$ for some $k\in \Nz_0$ we have that $C$ is   not angle-rigid  iff 
  condition 
  \eqref{suff} holds.
\end{corollary}

 Note that  $\calF  \cup  \calS_k$ does not exhaust all possible configurations.  More precisely, the situation is depicted in Figure
\ref{fig:corollary} below: as $k$ grows,
  $\calS_k $ covers all angle-rigid configurations and  
  configurations in $\calS_k$  not being angle-rigid but belonging  to $\calF$ (Proposition
  \ref{prop: necessary}). There exist not angle-rigid configurations
  which are not in $\calF$ (see Figure \ref{fig:cookie}), as well as
  configurations in $\calF$ (hence  not angle-rigid  by   definition)   that are not in $\calS_k$ for any  $k \in \Nz_0$ 
  (see left-most configuration in Figure~\ref{fig:three}).

\begin{figure}[ht]
\centering
\pgfdeclareimage[width=105mm]{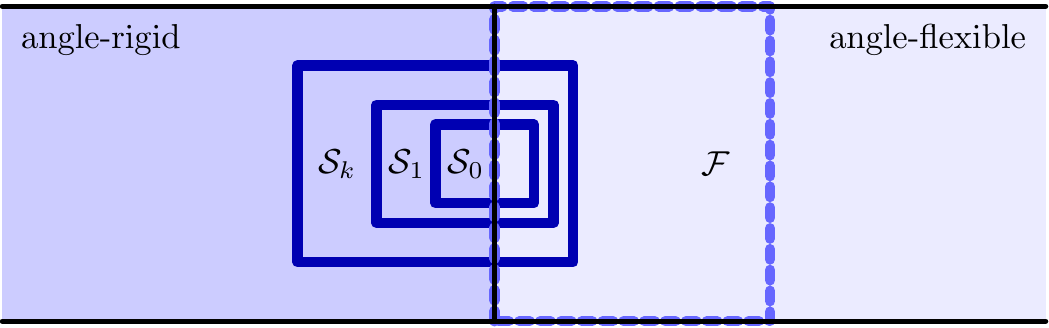}{corollary}  
\pgfuseimage{corollary} 
\caption{ The classes $\calF$ (dashed) and $\calS_k$ (solid) in
  relation with angle-rigidity.}
\label{fig:corollary}
\end{figure}

  \section{Characterization of  $0$-shear-resistant  cells} \label{sec:char}

Within the class $\calF  \cup   \calS_k$, 
Corollary \ref{cor} reduces the problem of assessing angle-rigidity to
that of verifying condition \eqref{suff}. Therefore, we are interested in checking if a configuration $C \in \Zz^{2n}$ lies in $\calF \cup  \calS_k$. As $C \in \mathcal{F}$ is equivalent to condition  \eqref{suff} by Proposition \ref{prop: F},   one is hence
left to  assess  if there exists $k \in \Nz_0$ such  that $C\in \calS_k$. 

%

To check if $C\in \calS_k$ can still be very demanding from an
 computational  viewpoint.  Consider again the example  in Figure \ref{fig:
  shearing}, where the configuration does not belong to $\calS_k$ for
any $k\in \Nz_0$ (and indeed is angle-flexible).  Adding  a
bonded point on the  top of the rightmost line makes the rightmost cell with face $f_2$ rigid.  This augmented configuration turns out to belong to $\mathcal{S}_k$ for some $k \in \Nz$. 

This in particular entails that, in general, belonging to class
$\mathcal{S}_k$ is a global property, and checking it requires 
the analysis of the whole configuration.

The focus of this section is on the  class   $\calS_0$.  This  is  special, for we are able to present complete
characterizations based
on {\it localized arguments}.


 We start by introducing some notation. 
 Let  a cell  $Z\in \mathcal{Z}(C)$ be given.   Referring to the discussion at the beginning of Section \ref{sec: cond}, we suppose that  $Z$ does not contain acyclic bonds. Therefore,  the
boundary of   its face  $f(Z)$ is a  (simple)  polygon.  We denote the points of the cell by $\lbrace \bp^1,\ldots,
\bp^m\rbrace$ with $|\bp^i - \bp^{i+1}| = 1$ for $i=1,\ldots,m$, where we understand the
 indices modulo $m$, if necessary. Note that the points are not
 necessarily pairwise different.  We say that a cell is \emph{simple}
 if  $|\bp^i - \bp^{j}| > 1$ for all $i,j$ with $|i-j| \ge 2$.  We refer to Figure \ref{fig:simple} for examples of nonsimple cells. In case
 that $Z$ is not simple, the bond  structure of $Z$ is
 delimited by a   minimal  (in the sense of set inclusion)   polygon of points of $Z$,  which we call the \emph{outer
   polygon}. Correspondingly, the points of $Z$ in the outer polygon
  form   a cell, which we call the \emph{outer cell}, denoted by
 $Z^{\rm out}$. Up to  a negligible set,  the set $f(Z^{\rm
   out}) \setminus f(Z)$ is the union of faces corresponding to other
 cells of the bond  structure.  These  cells are called the \emph{inner cells} corresponding to $Z$ and their collection is indicated by $\mathcal{I}_Z$. For simple cells $Z$, we have $Z  = Z^{\rm out}$ and  $\mathcal{I}_Z = \emptyset$. Note that $Z^{\rm out}$ and $\mathcal{I}_Z$ are simple cells.   In fact, if an inner cell were not simple, it could be decomposed into two smaller cells. If the outer cell were not simple, there would exist a polygon with less points of $Z$  delimiting $Z$.

\begin{figure}[ht]
\centering
\pgfdeclareimage[width=115mm]{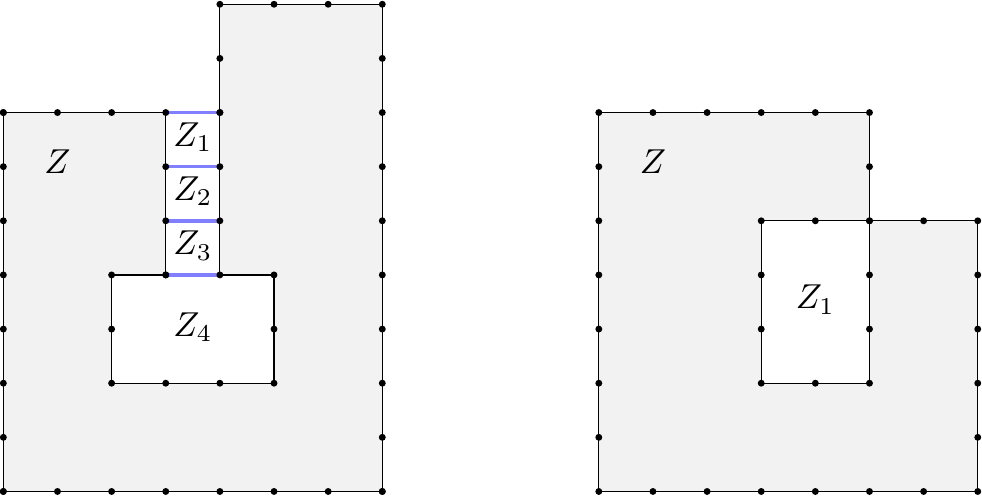}{simple}  
\pgfuseimage{simple} 
\caption{Examples of nonsimple cells. Here, $f(Z)$ is shaded and $Z_i\in {\mathcal I}_Z$
  are  inner   cells.}
\label{fig:simple}
\end{figure}

 We denote the vertices of  the  polygon $\partial f(Z)$   by $\lbrace \bv^1,\ldots,
\bv^n\rbrace$ for  some 
 $n \in 2\mathbb{N}$, ordered  counterclockwise.   (Here, $n$ is unrelated to the number of points in a configuration.)      We  understand the
 indices modulo $n$, if necessary.  Clearly, we have  $\lbrace \bv^1,\ldots,
\bv^n\rbrace \subset \lbrace \bp^1,\ldots,
\bp^m\rbrace$, where in general the inclusion is strict.   We decompose the set of  the
 vertices of the polygon 
 into the \emph{right- and left-set} 
\begin{align*}
I_{1,+}  (Z)  = \lbrace  \bv^i  \colon \,  \bv^{i+1} - \bv^i  \in
\mathbb{N} (1,0) \rbrace,\quad    I_{1,-}(Z) = \lbrace  \bv^i \colon \,  \bv^{i+1} -\bv^i  \in \mathbb{N} (-1,0) \rbrace 
\end{align*}
 and the \emph{up- and down-set}
\begin{align*}
I_{2,+}(Z) = \lbrace  \bv^i \colon \,  \bv^{i+1} - \bv^i  \in
\mathbb{N} (0,1) \rbrace,\quad    I_{2,-}(Z) = \lbrace  \bv^i \colon \,  \bv^{i+1} -\bv^i  \in \mathbb{N} (0,-1) \rbrace.
\end{align*}
 Note that we necessarily have that $ \#  I_{j,\pm}(Z) \geq 1$ for $j=1,2$. 

  With the aim of proving a characterization of $0$-shear-resistant cells, we address the case of simple cells first.  

\begin{proposition}[Characterization of $0$-shear-resistance,  simple cells]\label{lemma: 0-shear char}
A  simple  cell  $Z$  is $0$-shear-resistant  iff  
$\min\lbrace \# I_{1,+}(Z), \#I_{1,-}(Z) \rbrace  =  \min\lbrace \#
I_{2,+}(Z), \#I_{2,-}(Z) \rbrace  =  1$  iff the configuration consisting of $Z$ is angle-rigid.   
\end{proposition}
%

For the proof  of Proposition \ref{lemma: 0-shear char}, we will
make use of the elementary geometrical construction described in the
following lemma. 

\begin{lemma}[Geometric construction]\label{lemma: circles}
Let $\by^1,\by^2,\by^3,\by^4 \in  \lbrace 0 \rbrace \times \mathbb{R}
 \times \lbrace 0 \rbrace   $ be given. Suppose that  the second
 components of these vectors  satisfy  
$$ {\rm (a)} \ \ y^1_2 < y^4_2 \le y^2_2 < y^3_2 \quad \quad \quad \quad \text{or} \quad \quad \quad \quad {\rm (b)} \ \ y^1_2  \le y^3_2 < y^4_2 \le y^2_2.$$ 
Then, for $t>0$ small,  there exist $\tilde{\by}^3_t,\tilde{\by}^4_t
\in  \lbrace 0 \rbrace  \times  \mathbb{R}^2 $   with
$\tilde{\by}^3_t \not = {\by}^3 $, $\tilde{\by}^4_t \not = {\by}^4
$,  
$\tilde{\by}^3_t \to \by^3$, and $\tilde{\by}^4_t \to \by^4$ as $t \to 0$ such that
\begin{align}\label{eq: lengths the same}
 |\by^1 - \tilde{\by}^4_t| = |\by^1-\by^4|,  \ \    |\by^2-\tilde{\by}^3_t| = |\by^2-\by^3|, \ \  \text{ and } \ \  |\tilde{\by}^3_t-\tilde{\by}^4_t| = |\by^3-\by^4|.
\end{align}       
\end{lemma}

\begin{proof}
For simplicity, we only prove case (a). Case (b) follows along similar
lines with a different notational realization.  Given $t>0$, let
$\tilde{\by}^4_t = \by^4 +  (0, s(t), t)$  with 
 $t \mapsto s(t) < 0$ such
that   $|\by^1 - \tilde{\by}^4_t| =
|\by^1-\by^4|$ holds (note that  necessarily $s(t) 
= {\rm O}(t^2)$).   As a next step, we let
$\tilde{\by}^3_t=(0,x_2^t,x_3^t)$ where  $x_3^t$  solves
\begin{align}\label{eq: solution?}
y^2_2 + \sqrt{|\by^2 - \by^3|^2 -   (x^t_3)^2  } = y_2^4 +  s(t) + \sqrt{|\by^3 - \by^4|^2 -  (x^t_3-t)^2  }
\end{align}
and  $x_2^t$  is defined by  
$$ x^t_2   = y^2_2 + \sqrt{|\by^2 - \by^3|^2 -  (x^t_3)^2 }.$$
Note that 
the above equations imply that 
 $|\by^2-\tilde{\by}^3_t| = |\by^2-\by^3|$ and
$|\tilde{\by}^3_t-\tilde{\by}^4_t| = |\by^3-\by^4|$.  By checking
that $   x_3^t = {\rm O}(t)$ as $t\to 0$ we hence obtain the assertion.

One is therefore left to prove that equation   \eqref{eq:
  solution?} has at least one solution $ x_3^t $ with $ x_3^t = {\rm
  O}(t)$.  An elementary albeit  tedious computation 
allows to rewrite  \eqref{eq: solution?} as  $2 t x_3^t   = g_t( x_3^t )$, where 
\begin{align}\label{eq: solution?2}
 g_t( x_3^t ) := t^2 + s^2 +2(b+s)(b-a) + 2(a-b-s)\sqrt{b^2 - ( x_3^t )^2},
\end{align}
where  we have used the shorthand notation  $a:= |\by^3 -
\by^4|$, $b = |\by^2 - \by^3|$,  and $s=s(t)$.  Recall that $a \ge b$ by
assumption  and that $s<0$.   

 If  $a>b$, we obtain  $g_t(0)>0$ and $g_t(b)<0$  for $t$
(and thus $s$) sufficiently small. Therefore, by the  Intermediate
Value Theorem applied to the continuous function $x \mapsto -2tx
+ g_t(x)$   we find $ x_3^t  \in  (0,b)$   such that
\eqref{eq: solution?2} holds.  By observing that  $g_t$  is maximized
at $x=0$,  since $\max g_t = g_t(0)= {\rm O}(t^2 + s^2 + s) = {\rm
  O}(t^2)$, we also get $ x_3^t  = {\rm O}(t)$. 

 If  $a = b$, we find $g_t(0) > 0 $   and $2bt > g_t(b)= t^2 +
s^2$ for $t$ sufficiently small.  Hence,  also in this case \eqref{eq: solution?2} has a solution $ x_3^t $ with   $ x_3^t  = {\rm O}(t)$.
\end{proof}

 Having prepared the approximation tool of Lemma \ref{lemma:
  circles}, we can now proceed to the proof of Proposition \ref{lemma:
  0-shear char}.

\begin{proof}[Proof of  Proposition  \ref{lemma: 0-shear char}]
 Suppose that $Z$ is simple.   Assume first that  
$$\min\lbrace \# I_{1,+}, \#I_{1,-} \rbrace = \min\lbrace \# I_{2,+}
, \#I_{2,-} \rbrace= 1.$$
 (Here and in the sequel, we drop '$(Z)$' from  the notation for simplicity.)  We want to prove that  $Z$ is  angle-rigid.   After rotation,
reflection, and relabeling of the vertices, it is not restrictive to
suppose that  $\bv^1 \in  I_{1,+}  $, $\bv^2 \in I_{2,+}$, $\bv^{1+ 2i}
\in   I_{1,-}$ for $i=1,\ldots, n/2-1$, and   $\bv^{2+ 2i} \in I_{2,-}$
for $i=1,\ldots, n/2-1$,  see Figure \ref{fig:p}.   
\begin{figure}[ht]
\centering
\pgfdeclareimage[width=65mm]{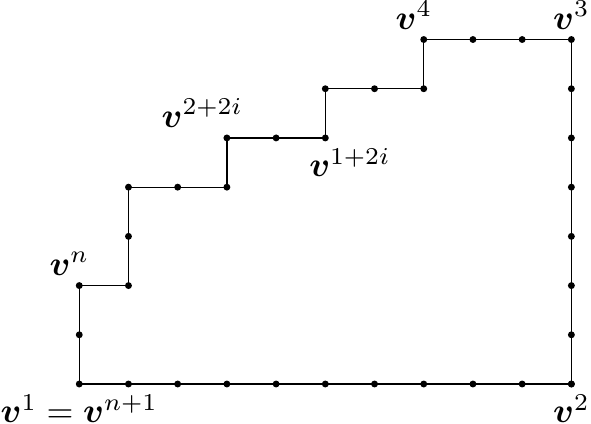}{p}  
\pgfuseimage{p} 
\caption{The points $\bv^i$. }
\label{fig:p}
\end{figure}

Let $\bvarphi$ be  an angle-preserving  mapping  on $Z$.   We may assume without restriction that
 $\bvarphi(\bv^i) = \bv^i$  for $i=1,2,3$. Recalling the
identification $\bv^{n+1} = \bv^1$ this particularly implies  that
$|\bvarphi(\bv^3) - \bvarphi(\bv^{n+1})| =  |\bv^3 - \bv^{ n+1}|$.  As
$\bv^{1+ 2i} \in   I_{1,-}$ for $i=1,\ldots, n/2-1$ and   $\bv^{2+ 2i}
\in I_{2,-}$ for $i=1,\ldots, n/2-1$, this is only possible if the
points $\lbrace \bvarphi(\bv^i)\rbrace_{i=3}^{n+1}$ are coplanar. 
Indeed, for any $4\le i \le n-1$, a rotation  of the points
$\bv^{i},\ldots,\bv^{n+1}$ around the axis containing the points
$\bv^i$ and $\bv^{i+1}$  decreases the distance between $\bv^3$ and
$\bv^{n+1}$.    By using \eqref{ar2} for $(\bv^n,\bv^1,\bv^2)$ and  $(\bv^2,\bv^3,\bv^4)$ this yields that $\lbrace \bvarphi(\bv^i)\rbrace_{i=3}^{n+1}$ are contained in $\mathbb{Z}^2 \times \lbrace 0 \rbrace$, and therefore   $\bvarphi$  is the identity.  This proves  that $Z$ is angle-rigid.

Next, if $Z$ is assumed to be angle-rigid,  the validity of Condition (i) of Definition \ref{def:
  shear} clearly follows, i.e., $Z$ is $0$-shear resistant.

 Eventually, we assume that $Z$ is $0$-shear resistant and  show that $\min\lbrace \# I_{1,+}, \#I_{1,-} \rbrace = \min\lbrace \# I_{2,+} , \#I_{2,-} \rbrace= 1.$  We show the contrapositive:  suppose that $\min\lbrace \#
I_{2,+} , \#I_{2,-} \rbrace \ge 2$ and we prove that $Z$ is not $0$-shear-resistant  (the case  $\min\lbrace \#
I_{1,+} , \#I_{1,-} \rbrace \ge 2$ is analogous). 
Up to  a  cyclic relabeling of the vertices, we  can reduce
 the problem to  the following two different cases: we have
$\lbrace \bx^1,\ldots,\bx^4\rbrace$  with $\bx^i = \bv^{m(i)}$ for
$i=1,\ldots,4$, where $m\colon \lbrace 1,\ldots,4\rbrace \to \lbrace
1,\ldots,n\rbrace$ is strictly increasing, such that  one of the
following two cases hold 
\begin{align}
 &  {\rm (a)} \quad \bx^1, \bx^2 \in I_{ 2 ,+} \text{ and } \bx^3, \bx^4
   \in I_{ 2 ,-} \text{ with }  x^1_2 < x^2_2 < x^3_2 \text{ and } x^1_2
   < x^4_2 < x^3_2,\label{cases1}\\
& {\rm (b)} \quad \bx^1, \bx^3 \in I_{ 2 ,+} \text{ and }    \bx^2, \bx^4
\in I_{ 2 ,-} \text{ with }  x^1_2  \le x^3_2 < x^2_2 \text{ and }
x^3_2 < x^4_2, \label{cases2}
\end{align}
 see Figure \ref{fig:x}. 
  \begin{figure}[ht]
\centering
\pgfdeclareimage[width=100mm]{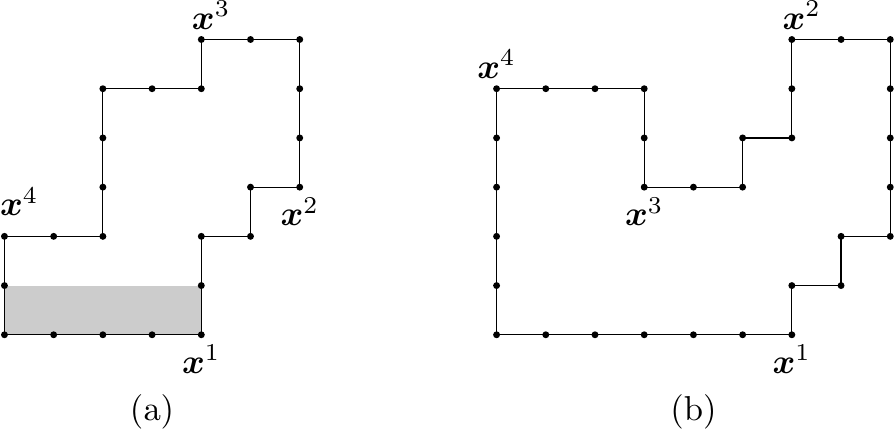}{x}  
\pgfuseimage{x} 
\caption{ The two cases from \eqref{cases1}--\eqref{cases2}. The shaded
  region is the paraxial rectangle mentioned in the proof.}
\label{fig:x}
\end{figure}

We treat case (a) first.   After reflection of the cell along the
$x_2$-axis  and interchanging the labels for $\bx^2$ and $\bx^4$,  we may suppose that $x^4_2 \le  x^2_2$.  We denote the
orthogonal projection of $\bx^i$  onto $ \lbrace 0 \rbrace \times
\Rz \times \lbrace 0 \rbrace $ by   $\by^i$ and note that $y^1_2 <
y_2^4 \le y_2^2 < y_2^3 $. We apply  Lemma \ref{lemma: circles}(a) and
obtain two points $\tilde{\by}^3_t, \tilde{\by}^4_t \in  \lbrace 0 \rbrace \times  \mathbb{R}^2$
for $t>0$ small such that \eqref{eq: lengths the same} holds. By
$R^{23}_t$ we denote the  rotation in $\Rz^3$ with axis parallel
to the $x_1$-axis which  leaves $\bx^2$ unchanged and moves $\bx^3$ to $\bx^3 + \tilde{\by}^3_t - \by^3_t$.  By $R^{34}_t$ we denote the  rotation in $\Rz^3$ with axis parallel
to the $x_1$-axis which moves $\bx^3$ to $\bx^3 + \tilde{\by}^3_t- \by^3$ and moves $\bx^4$ to $\bx^4 + \tilde{\by}^4_t - \by^4$. Finally,  by $R^{14}_t$ we denote the  rotation in $\Rz^3$ with axis parallel
to the $x_1$-axis  which leaves $\bx^1$ unchanged and moves $\bx^4$ to $\bx^4 + \tilde{\by}^4_t - \by^4$.  Note that these isometries exist due to \eqref{eq: lengths the same}.   We   now define  $\bvarphi$  on the vertices $\lbrace \bv^1,\ldots,\bv^n \rbrace$ by
\begin{align}\label{eq: transformation}
\bvarphi(\bv^i) = \begin{cases}   \bv^i & \text{if }  m(1) \le i \le m(2), \\   R^{23}_t(\bv^i) & \text{if }  m(2) \le i \le m(3), \\ R^{34}_t(\bv^i) & \text{if }  m(3) \le i \le m(4), \\ R^{14}_t(\bv^i) & \text{if }  m(4) \le i \le m(1). \end{cases} 
\end{align}
(Note that the definition is consistent for $i = m(j)$ for
$j=1,\ldots,4$.)  This mapping can be naturally extended to the points $\lbrace \bp^1,\ldots,\bp^m\rbrace$ of the cell $Z$: given a point $\bp^j \notin \lbrace \bv^1,\ldots,\bv^n \rbrace$ lying on the segment between $\bv^i$ and $\bv^{i+1}$, we set $\bvarphi(\bp^j) = R_t(\bp^j )$, where  $R_t$ denotes the isometry under which both $\bv^i$ and $\bv^{i+1}$ are moved (i.e., the identity or $R^{23}_t,  R^{34}_t,  R^{14}_t$).  By construction we can check that $\bvarphi$
 is angle-preserving.  Indeed, we observe that all bonds of $Z$ only join consecutive points $\bp^i$ and $\bp^{i+1}$, $i=1,\ldots,m$, since $Z$ is simple. Then,   the fact that the
mappings $R^{23}_t$, $R^{34}_t$, and $R^{14}_t$ are isometries implies
\eqref{ar1}.  In order to check \eqref{ar2},  we additionally
use that the isometries are  actually rotations about axes
parallel to the $x_1$-axis  and all bonds of $Z$ are of course either parallel
or orthogonal to $x_1$.  
Moreover, since $R^{14}_t$ is not the identity, we find that the
condition in Definition \ref{def: shear}(i) is violated for at least
one paraxial rectangle  (shaded region in Figure \ref{fig:x} left).

The case (b) is very similar and we only indicate the necessary
adaptations.  After reflection of the cell along the $x_2$-axis, we
may suppose that    $x^4_2 \le  x^2_2$.   As before, we indicate by $\by^i$
the orthogonal projection of $\bx^i$  onto $\lbrace 0 \rbrace \times
  \Rz \times \lbrace 0 \rbrace $  and note that $y^1_2 \le y_2^3 <
  y_2^4 \le y_2^2 $. We now apply Lemma \ref{lemma: circles}(b) and
  argue similarly as before to define  $\bvarphi$. 
\end{proof}

 We now proceed with the characterization of $0$-shear-resistance for
 general cells,  i.e., for possibly not simple cells.  To this end, recall the notions introduced at the beginning of Section \ref{sec:char}.

\begin{proposition}[Characterization of $0$-shear-resistance]\label{lemma: 0-shear char-gen}
A cell $Z\in \mathcal{Z}(C) $ is $0$-shear-resistant if and only if
\begin{align*}
&Z \notin \mathcal{F} \ \  \text{and} \nonumber\\
& \min\lbrace \# I_{1,+}(Z'), \#I_{1,-}(Z') \rbrace  =  \min\lbrace \#
I_{2,+}(Z'), \#I_{2,-}(Z') \rbrace  =  1 \quad \forall \, Z' \in
\lbrace Z^{\rm out}\rbrace \cup \mathcal{I}_Z.
\end{align*}
\end{proposition}

\begin{proof}
We start with a preliminary  observation.    Consider a cell in $Z' \in \lbrace
Z^{\rm out}\rbrace \cup \mathcal{I}_Z$. Then, each cell $\tilde{Z} \in
\lbrace Z^{\rm out}\rbrace \cup \mathcal{I}_Z$ with $\tilde{Z} \neq
Z'$ and  $Z' \cap \tilde{Z} \neq \emptyset$ shares at most one bond
with $Z'$. Consequently, for  each connected component $\widehat{C}$
of $Z \setminus Z'$ there is at least one and at most two points in
$Z'$ connected by a bond in $Z$  to $\widehat{C}$. (If there are two
points, these stay at  distance $1$.)  This directly comes from the
definition  of ${\mathcal I}_Z$ and $Z^{\rm out}$. In fact,  $f(Z)$ and $f(Z')$ for $Z' \in 
{\mathcal I}_Z$  corresponds to a  partition of $f(Z^{\rm out})$,   and  all segments in the bond structure which are not contained in $\partial f(Z)$ have length $1$, see Figure \ref{fig:simple}.

We now start with the actual proof.  Assume first that $Z \notin \mathcal{F}$ and   
$$\min\lbrace \# I_{1,+}(Z'), \#I_{1,-}(Z') \rbrace = \min\lbrace \# I_{2,+}(Z')
, \#I_{2,-}(Z') \rbrace= 1 \quad  \forall \, Z' \in \lbrace Z^{\rm out}\rbrace \cup \mathcal{I}_Z.$$
Our goal is to show that $Z$ is $0$-shear-resistant. To this end, let $\bvarphi$ be an angle-preserving   mapping  of $Z$. As each cell in $\lbrace Z^{\rm out}\rbrace \cup \mathcal{I}_Z$ is simple, Proposition \ref{lemma: 0-shear char} yields that the points of each cell in $\lbrace Z^{\rm out}\rbrace \cup \mathcal{I}_Z$ remain coplanar under  $\bvarphi$.   Moreover, for each $Z' \in  \mathcal{I}_Z$ and each $\bx \in
Z'$, we find a path  $(\bq^1,\ldots, \bq^l)$  in $Z$ connecting $\bx$
with a point $\by \in Z^{\rm out}$ such that, whenever $\bq^{i-1} \in
\tilde{Z}' \setminus \tilde{Z}''$, $\bq^{i} \in \tilde{Z}' \cap
\tilde{Z}''$, and $\bq^{i+1} \in \tilde{Z}'' \setminus \tilde{Z}'$ for
two cells $\tilde{Z}', \tilde{Z}'' \in \lbrace Z^{\rm out}\rbrace \cup
\mathcal{I}_Z$, then $\bq^{i+1} - \bq^i = \bq^i- \bq^{i-1}$. In fact,
otherwise we would necessarily find an axis such that \eqref{suff}
holds. Here, we use the preliminary observation that each cell
$\tilde{Z}'$ contains at most two points which are bonded to the
connected component $\widehat{C}$ of $Z \setminus \tilde{Z}'$ with
$Z^{\rm out} \setminus \tilde{Z}' \subset \widehat{C}$, where, if two
points exist, they have distance $1$ and thus lie on one axis.  This
property, however, contradicts $Z \notin \mathcal{F}$ and
Proposition~\ref{prop: F}.   This can be visualized in Figure
\ref{fig:simple}, left. Take $\tilde Z'=Z_3$, $\tilde Z''=Z_4$ and let  $\boldsymbol
q^i$  be  any of the two  points in $Z_3 \cap Z_4$. Then,  $\bq^{i+1} -
\bq^i \not= \bq^i- \bq^{i-1}$.  Indeed, \eqref{suff} holds for the axis given by  $Z_3 \cap Z_4$, and the configuration
turns out to be  a folding  around this axis.  


Now consider some $\bq^i \in \tilde{Z}' \cap \tilde{Z}''$ as
above. Since $\bq^{i+1} - \bq^i = \bq^i- \bq^{i-1}$, \eqref{ar2} and
the fact that the points of each of the two cells $ \tilde{Z}'$ and
$\tilde{Z}''$  remain coplanar under  $\bvarphi$, 
imply that the points in  $ \tilde{Z}' \cup \tilde{Z}''$ remain
coplanar under  $\bvarphi$.  A  successive application of this argument yields that the points in $Z^{\rm out} \cup Z'$ remain coplanar under  $\bvarphi$  for every $Z' \in \mathcal{I}_Z$. Eventually, this shows that all points of $Z$ remain coplanar under  $\bvarphi$,   and thus $Z$ is  $0$-shear-resistant.

We now address the reverse implication: to this end, we first suppose
that there exists $Z' \in \lbrace Z^{\rm out}\rbrace \cup
\mathcal{I}_Z$ such that $$\min\lbrace \# I_{1,+}(Z'), \#I_{1,-}(Z')
\rbrace >1\ \ \text{or} \  \ \min\lbrace \#
I_{2,+}(Z'), \#I_{2,-}(Z') \rbrace  >  1.$$
 As $Z'$ is simple, we can apply  Proposition \ref{lemma: 0-shear char}  to find a nontrivial angle-preserving   mapping  $\tilde{\bvarphi}$ of $Z'$ such that the
condition in Definition \ref{def: shear}(i) is violated.  It now
suffices to check that  $\tilde{\bvarphi}$ can be extended to an
angle-preserving   mapping  on the whole  of $Z$. To this end,
denote the points of $Z'$ by   $\lbrace \bp^1,\ldots,\bp^m\rbrace$,
and  fix a connected component $\widehat{C}$ of $Z \setminus Z'$.  In
view of our preliminary observation, we need to address two cases: (a)
there is exactly one point $\bp^i \in Z'$ bonded to $\widehat{C}$ and
(b) there are exactly two consecutive points $\bp^{i}, \bp^{i+1} \in
Z'$ bonded to $\widehat{C}$. In case (a), the three points $\bp^{i-1},
\bp^i, \bp^{i+1}$ are contained in two lines and thus moved by the
same isometry   under $\tilde{\bvarphi} $,  see the
construction in \eqref{eq: transformation}. We denote this isometry by
$R_t$ and set  $\bvarphi(\bx) = R_t(\bx)$ for all $\bx \in
\widehat{C}$. In case (b), the four points $\bp^{i-1}, \bp^i,
\bp^{i+1}, \bp^{i+2}$ are contained in two lines and thus, as before,
they are  moved by a single isometry $R_t$    under $\tilde{\bvarphi} $.  We again set  $\bvarphi(\bx) = R_t(\bx)$ for all $\bx \in \widehat{C}$. This construction leads to an angle-preserving   mapping  $\bvarphi$ defined on $Z$, which coincides with $\tilde{\bvarphi}$ on $Z'$. This shows that $Z$ is not  $0$-shear-resistant.

It remains to address the case  $$Z \in \mathcal{F} \ \text{and} \ \min\lbrace \# I_{1,+}(Z'), \#I_{1,-}(Z') \rbrace  =  \min\lbrace \#
I_{2,+}(Z'), \#I_{2,-}(Z') \rbrace  =  1$$
 for all $Z' \in \lbrace Z^{\rm out}\rbrace \cup
\mathcal{I}_Z$. In this case,  by Proposition \ref{prop: F} we find a nontrivial angle-preserving mapping  $\bvarphi$ of $Z$ of the form indicated in Definition \ref{def:s}. Now, we need to check that  the condition in Definition \ref{def: shear}(i) is violated for at least
one paraxial rectangle. In fact, by Proposition \ref{lemma: 0-shear
  char} we get that all points of $Z^{\rm out}$ are coplanar under
$\bvarphi$. On the other hand, by construction  necessarily there
exists  a bond of $Z$ in the open region $f(Z^{\rm out})$ which does not remain coplanar to $Z^{\rm out}$ under $\bvarphi$. Therefore, there exists a paraxial rectangle in $Z$ which does not remain coplanar under $\bvarphi$.      
\end{proof}

\section*{ Acknowledgements}
 This work is partially supported 
 by the FWF-DFG grant I\,4354, the FWF
grants F\,65, I\,5149, and P\,32788, by the Vienna Science and Technology Fund (WWTF)
through project MA14-009, and by the OeAD-WTZ
project CZ 01/2021. This work was supported by the Deutsche
Forschungsgemeinschaft (DFG, German Research Foundation) under
Germany's Excellence Strategy EXC 2044 -390685587, Mathematics
M\"unster: Dynamics--Geometry--Structure. The authors want to thank
Jan Legersk\'y for his helpful comments on a previous version of the manuscript.

\bibliographystyle{alpha}


\end{document}